%% file: UMP_arxiv_2022.tex
\documentclass[12pt]{article}

\usepackage[english]{babel}
\usepackage{amsmath, amsfonts, amssymb, amsthm}
\usepackage{url}
\usepackage[misc,geometry]{ifsym}
\usepackage{enumitem}
\usepackage{fullpage}
\usepackage[T2A]{fontenc}
\usepackage[utf8]{inputenc}
\usepackage{multirow}

\usepackage[pdftex]{hyperref}
\usepackage{cleveref}

\usepackage{algorithm}
\usepackage{algorithmic}
\usepackage{caption}

\newcommand{\e}{\varepsilon}
\newcommand{\la}{\langle}
\newcommand{\ra}{\rangle}

\def\beq{\begin{equation}}
\def\eeq{\end{equation}}
\def\ba{\begin{array}}
\def\ea{\end{array}}
\def\beann{\begin{eqnarray*}}
\def\eeann{\end{eqnarray*}}
\def\bea{\begin{eqnarray}}
\def\eea{\end{eqnarray}}

\newtheorem{Ex}{Example}

\newtheorem{Lm}{Lemma}

\newtheorem{Th}{Theorem}

\newtheorem{Cor}{Corollary}

\theoremstyle{remark}
\newtheorem{Rem}{Remark}

\theoremstyle{definition}
\newtheorem{Def}{Definition}

\def\BLm{\begin{Lm}}
\def\ELm{\end{Lm}}

\input{ex_shared}

\ifpdf
\hypersetup{
 pdftitle={Generalized Mirror Prox},
 pdfauthor={F. Stonyakin, A. Gasnikov, P. Dvurechensky, A. Titov, M. Alkousa}
}
\fi

\usepackage[font=small,labelfont=bf]{caption}
\usepackage[]{subfigure}
\usepackage{float}
\usepackage{graphicx}

\makeatletter
\let\@fnsymbol\@arabic
\makeatother

\begin{document}

\maketitle

\begin{abstract}
 We introduce an inexact oracle model for variational inequalities with monotone operator, propose a numerical method, which solves such variational inequalities, and analyze its convergence rate. As a particular case, we consider variational inequalities with H\"older-continuous operator and show that our algorithm is universal. This means that, without knowing the H\"older exponent and H\"older constant, the algorithm has the least possible in the worst-case sense complexity for this class of variational inequalities. We also consider the case of variational inequalities with strongly monotone operator and generalize the algorithm for variational inequalities with inexact oracle and our universal method for this class of problems. Finally, we show, how our method can be applied to convex-concave saddle point problems with H\"older-continuous partial subgradients.
 \textbf{AMS: }65K15
, 90C33
, 90C06
, 68Q25
, 65Y20
, 68W40
, 58E35

\end{abstract}

\section{Introduction} \label{S:Intro}
This paper is devoted to Minty \cite{mintyVI} (or weak \cite{nesterov2007dual}) variational inequalities with a monotone and continuous operator. Variational inequalities with monotone operators are closely connected with convex optimization problems and convex-concave saddle point problems. In the former case, the operator is just the subgradient of the objective function, and in the latter case the operator is composed from partial subgradients of the objective in the saddle point problem. Studying variational inequalities is important also for equilibrium and complementarity problems \cite{harker1990finite,facchinei2007finite} and saddle point problems has become an important part of research in machine learning \cite{arjovsky2017wasserstein,kniaz2021adversarial}.

Our focus here is on numerical methods for such problems, their convergence rate and complexity estimates. Significant contribution to the development of numerical methods for solving variational inequalities was made in 1970's, when the extragradient method was proposed in \cite{korpelevich1976extragradient}. More recently, \cite{nemirovski2004prox} proposed a non-Euclidean variant of this method, called Mirror Prox algorithm, which can be applied for Lipschitz continuous operators. 

Different methods with similar complexity were also proposed in \cite{solodov1999hybrid,auslender2005interior,nesterov2007dual,monteiro2010complexity,koshal2011multiuser,gasnikov2019adaptive}. Besides that, in \cite{nesterov2007dual}, Nesterov proposed a method for variational inequalities with bounded variation of the operator, i.e. with non-smooth operator. He raised also a question, whether it is possible to propose a method, which automatically "adjusts to the actual level of smoothness of the current problem instance". One of the goals of this paper is to propose such an algorithm.

To this aim, we consider a more general class of operators being so-called H\"older-continuous. This class covers both the case of operators with bounded variation and Lipschitz-continuous operators. Variational inequalities with\\ H\"older-continuous monotone operator were already considered in \cite{nemirovski2004prox}, where a special choice of the stepsize for the Mirror Prox algorithm led to the optimal complexity for this class of problems; see \cite{nemirovsky1983problem}. The authors of \cite{dang2015convergence} consider variational inequalities with non-monotone H\"older-continuous operator. Unfortunately, both papers use H\"older constant and exponent to define the stepsize of their methods. This is in contrast to optimization, where so called universal algorithms were proposed, which do not use the information about the H\"older exponent and  H\"older constant; see \cite{nesterov2015universal,lan2015generalized,gasnikov2018universal,baimurzina2019universal,dvurechensky2017gradient,dvurechensky2017universal,guminov2017universal,guminov2019accelerated,nesterov2020primal-dual}. In this paper, we propose a universal method for variational inequalities with H\"older-continuous monotone operator. We also generalize this method for the case of variational inequalities with strongly monotone operator. Such problems were considered in \cite{nestrov2011strongly_VI}, but only for the case of  Lipschitz-continuous operator with known Lipschitz constant.

On the other hand, as it was shown for optimization problems in \cite{devolder2014first,nesterov2015universal}, universal methods have a natural connection with methods for smooth problems with inexact oracle. Namely, it can be shown that a function with H\"older-continuous subgradient can be considered as a Lipschitz-smooth function with inexact oracle. Despite that there are many works on optimization methods with inexact oracle, see e.g. \cite{aspremont2008smooth,devolder2014first,dvurechensky2016stochastic,gasnikov2016stochasticInter,bogolubsky2016learning,dvurechensky2017universal,cohen2018acceleration,aybat2018robust,stonyakin2019gradient}, we 
are not aware of any extensions of these non-stochastic definitions of inexact oracle to the variational inequality setting and methods for variational inequalities with inexactly given operator, except stochastic case. By this paper, we introduce a theory of methods for variational inequalities with deterministic inexact oracle, see also the follow-up work \cite{stonyakin2020inexact}.

\section{Preliminaries}\label{S:prel}
We start with the general notations, problem statement, and description of proximal setup.
Let $E$ be a finite-dimensional real vector space and $E^*$ be its dual. We denote the value of a linear function $u \in E^*$ at $x\in E$ by $\langle u, x \rangle$. Let $\|\cdot\|$ be a general norm on $E$, $\|\cdot\|_{*}$ be its dual, defined by \\$\|u\|_{*} := \max\limits_{x} \big\{ \langle u, x \rangle, \| x \| \leq 1 \big\}$. We use $\nabla f(x)$ to denote any subgradient of a function $f$ at a point $x \in {\rm dom} f$.

The main problem, we consider, is the following Minty variational inequality (VI)
\begin{equation}
\text{Find} \quad x_* \in Q : \quad \langle g(x), x_* - x \rangle \leq 0, \quad \forall x \in Q
\label{eq:PrSt}
\end{equation}
where $Q$ is convex (non-necessary compact) subset of finite-dimensional vector space $E$, $g: Q \to E^*$ is continuous, monotone operator
$$
\langle g(x) - g(y) , x - y \rangle \geq 0, \quad x,y \in Q,
$$
satisfying H\"older condition on $Q$, i.e., for some $\nu \in [0,1]$  and $L_{\nu} \geq 0$,
\begin{equation}
\label{eq:GHoeldDef}
\|g(x) - g(y)\|_* \leq L_{\nu}\|x-y\|^{\nu}, \quad x,y \in Q.
\end{equation}
We refer to $\nu$ as H\"older exponent and to $L_\nu$ as H\"older constant. We assume that the variational inequality \eqref{eq:PrSt} has a solution.
Under the assumption of continuity and monotonicity of the operator $g$, the problem \eqref{eq:PrSt} is equivalent to a Stampacchia \cite{mintyVI} (or strong \cite{nesterov2007dual}) variational inequality, in which the goal is to find $x_* \in Q $ such that 
\begin{equation}
\label{eq:strongVI}
\langle g(x_*), x_* - x \rangle \leq 0, \quad \forall x \in Q.
\end{equation}
Following \cite{nesterov2007dual,antonakopoulos2021adaptive}, to assess the quality of a candidate solution $\widehat{x}$, we use a convex non-empty compact subset $C$ of the set $Q$ and the following  restricted  gap (or merit) function
	\begin{equation}
	\label{eq:gap_func}
	{\rm Gap}_C(\widehat{x}) = \max_{u \in C} \langle g(u), \widehat{x}  - u \rangle .
	\end{equation}
	Proposition 1 in \cite{antonakopoulos2021adaptive} states that ${\rm Gap}_C(\widehat{x}) \geq 0 $ whenever $\widehat{x} \in C$ and if ${\rm Gap}_C(\widehat{x})= 0 $ and $C$ contains a neighborhood of $\widehat{x}$, then $\widehat{x}$ is a solution of \eqref{eq:strongVI}.
	This motivates our goal that is to find an approximate solution of the problem, that is, a point $\widehat{x} \in Q$ such that, for some $\varepsilon >0$
	\begin{equation}
	\label{eq:Appr}
	{\rm Gap}_C(\widehat{x}) = \max_{u \in C} \langle g(u), \widehat{x}  - u \rangle \leq \varepsilon.
	\end{equation}

As already mentioned, \cite{nemirovski2004prox} proposed Mirror Prox algorithm under the assumption of compactness of $Q$ and $L_1$-Lipschitz continuity of the operator, i.e., $g$ satisfying \eqref{eq:GHoeldDef} with $\nu=1$ and $L_1$. This method has complexity $O\left(\frac{L_1R^2}{\varepsilon}\right)$, where $R$ characterizes the diameter of the set $Q$ and $\varepsilon$ is the desired accuracy. By complexity we mean the number of iterations of an algorithm to find a point $\widehat{x} \in Q$ such that \eqref{eq:Appr} holds. For the case of variational inequalities with bounded variation of the operator $g$, i.e., $g$ satisfying \eqref{eq:GHoeldDef} with $\nu=0$ and $L_0$
\cite{nesterov2007dual} proposed a method with complexity $O\left(\frac{L_0^2R^2}{\varepsilon^2}\right)$.
The method for variational inequalities with H\"older-continuous monotone operator \cite{nemirovski2004prox} has the complexity
$$
O\left(\left(\frac{L_{\nu}}{\varepsilon}\right)^{\frac{2}{1+\nu}} R^2 \right),
$$
which is optimal for the case of $\nu = 1$ and for the case of $\nu = 0$ \cite{ouyang2021lower,nemirovsky1983problem}.

Next we give several definitions, which are necessary for introducing the method. We choose a {\it prox-function} $d(x)$, which is continuous and convex on $Q$, and also is
\begin{enumerate}
	\item continuously differentiable at the relative interior of $Q$;
	\item $1$-strongly convex on $Q$ with respect to $\|\cdot\|$, i.e., for any $x \in Q^0, y \in Q$ $d(y)-d(x) -\langle \nabla d(x) ,y-x \rangle \geq \frac12\|y-x\|^2$.
\end{enumerate}
Without loss of generality, we assume that $\min\limits_{x\in Q} d(x) = 0$.

We define also the corresponding {\it Bregman divergence} $$V[z] (x) := d(x) - d(z) - \langle \nabla d(z), x - z \rangle, \ x \in Q, z \in Q^0.$$ Standard proximal setups, i.e., Euclidean, entropy, $\ell_1/\ell_2$, simplex, nuclear norm, spectahedron can be found in \cite{ben-tal2015lectures}.
Below we use Bregman divergence in so-called {\it prox-mapping}
\begin{equation}
\min_{x \in Q} \left\{\langle g,x \rangle + M V[\bar{x}](x)		 \right\} ,
\label{eq:PrMap}
\end{equation}
where $M >0 $, $\bar{x} \in Q^0$, $g \in E^*$ are given. We allow this problem to be solved inexactly in the following sense inspired by \cite{ben-tal2015lectures}.
Assume that we are given $\delta_{pu} >0$, $M >0 $, $\bar{x} \in Q^0$, $g \in E^*$. Further, we assume that for an arbitrary $\delta_{pc} >0$, we can calculate $\tilde{x} = \tilde{x}(\bar{x},g,M,\delta_{pc},\delta_{pu}) \in Q^0$ such that
	\begin{equation}
	\left\langle g + M\left[\nabla d(\tilde{x}) - \nabla d(\bar{x}) \right], u - \tilde{x} \right\rangle \geq - \delta_{pc}-\delta_{pu}, \quad \forall u \in Q.
	\label{eq:InPrMap}
	\end{equation}
	We call the point $\tilde{x}$ an {\it inexact prox-mapping} and write
\begin{equation}
\tilde{x} :=  {\mathop {\arg \min }\limits_{x\in Q}}^{\delta_{pc}+\delta_{pu}}\left\{\langle g,x \rangle + M V[\bar{x}](x)	 \right\}.
\label{eq:InPrMap1}
\end{equation}
Here $\delta_{pu}$ denotes the error of the prox-mapping, which is not controlled, and $\delta_{pc}$ denotes the error of the prox-mapping, which can be controlled and made as small as desired.

\section{Inexact Oracle for Variational Inequalities}\label{S:IO}
Our goal is to consider, in a unified manner, VIs with H\"older-continuous operator and VIs with inexact values of the operator. This can be done by considering H\"older-continuous operator as a particular case of Lipschitz-continuous operator with some inexactness. Thus, we introduce the following definition of inexact oracle for the operator $g$.

\begin{Def}\label{Def:IO}
	\textit{
		Assume that for some $\delta_u > 0$ (uncontrolled error) and for any number $\delta_c > 0$ (controlled error) there exists a constant $L(\delta_c) \in ]0, +\infty[$ such that, for any points $x,y \in Q$, one can calculate $\tilde{g}(x,\delta_c,\delta_u)$ and\\ $\tilde{g}(y,\delta_c,\delta_u) \in E^*$ satisfying
		\begin{align}
		\hspace{-1em}\langle \tilde{g}(y,\delta_c,\delta_u) - \tilde{g}(x,\delta_c,\delta_u), y - z \rangle & \leq \frac{L(\delta_c)}{2}\left(\|y-x\|^2 + \|y-z\|^2\right) + \delta_c+\delta_u, 	\label{eq:g_or_def} \\ 
		\langle \tilde{g}(y,\delta_c,\delta_u) - g(y), y - z \rangle & \geq - \delta_u, \quad \forall z \in Q. \label{eq:g_or_deff}
		\end{align}
		Then, the operator $\tilde{g}(\cdot,\delta_c,\delta_u)$ is called inexact oracle for the operator $g$.}
\end{Def}

In this definition, $\delta_c$ represents the error of the oracle, which we can control and make as small as we would like to. On the opposite, $\delta_u$ represents the uncontrolled error, which can be understood as an error in the problem data, for example, when $g$ is given as a solution to an auxiliary problem. We notice also that if the inequality \eqref{eq:g_or_def} holds for some $L(\delta_c)$, then it holds also for any $\widetilde{L}(\delta_c) \geq L(\delta_c)$.

Example \ref{Ex:HoldInexact} below shows that this definition satisfies our goal of covering both the case of H\"older-continuous operator and the case of inexact values of the operator. The following technical lemma is the main clue for this example.

\begin{Lm}\label{Lm:HoldToLip}
	Let $a,b,c \geq 0$, $\nu \in [0,1]$. Then, for any $\delta > 0$,
	$$
	ab^{\nu}c \leq \left(\frac{1}{\delta}\right)^{\frac{1-\nu}{1+\nu}} \frac{a^{\frac{2}{1+\nu}}}{2} \left(b^2+c^2\right) + \frac{\delta}{2}.
	$$
\end{Lm}
The proof of this lemma is given in the Appendix A.


\begin{Ex}\textit{(H\"older-continuous operator with inexact values on a bounded set).}\label{Ex:HoldInexact}
	Let us assume that:
	\begin{enumerate}
		\item The operator $g(x)$ is H\"older-continuous on $Q$, i.e., satisfies \eqref{eq:GHoeldDef}.
		\item\label{eq:def_in_2_ex3} The set $Q$ is bounded with $\max_{x, y \in Q} \|x-y\| \leq D$.
		\item\label{eq:def_in_3_ex3} There exist $\bar{\delta}_u > 0$ and at any point $x \in Q$, we can calculate approximation $\bar{g}(x)$ for $g(x)$ such that $\|\bar{g}(x)-g(x)\|_* \leq \bar{\delta}_u$.
	\end{enumerate}
	Then, for any $z \in Q$,
	\begin{align*}
	\langle \bar{g}(y) - \bar{g}(x) , y-z \rangle &= \langle \bar{g}(y) - g(y) , y-z \rangle - \langle \bar{g}(x) - g(x) , y-z \rangle + \langle g(y) - g(x) , y-z \rangle \\
	& \leq  2 \bar{\delta}_u D + \|g(y) - g(x) \|_{*} \|y-z\| \leq 2 \bar{\delta}_u D + L_{\nu} \|y-x\|^{\nu} \|y-z\| \\
	& \leq  2 \bar{\delta}_u D + \frac{1}{2}\left(\frac{1}{\delta}\right)^{\frac{1-\nu}{1+\nu}} L_{\nu}^{\frac{2}{1+\nu}} \left(\|x-y\|^2+\|y-z\|^2\right) + \frac{\delta}{2},
	\end{align*}
	where Lemma \ref{Lm:HoldToLip} was used to get the last inequality.\\
	Thus, we can set $\delta_u = 2 \bar{\delta}_u D$,  $\delta_c =\frac{\delta}{2}$, and $L(\delta_c) = \left(\frac{1}{2\delta_c}\right)^{\frac{1-\nu}{1+\nu}} L_{\nu}^{\frac{2}{1+\nu}}$ to get \eqref{eq:g_or_def}. 
	Further,
	$$
	|\langle \overline{g}(y)-g(y),y-z \rangle|\leq \| \overline{g}(y)-g(y) \|_* \|y-z\| \leq \overline{\delta}_{u}D = \frac{1}{2}\delta_u < \delta_u
	$$
	and we have $
	\langle \overline{g}(y)-g(y),y-z \rangle >  - \delta_u$, which is \eqref{eq:g_or_deff}.
\end{Ex}	


\begin{Ex}\textit{(Connection with $(\delta,L)$-oracle in optimization).}\label{ExampVI2.5}
	Let convex function $f:\,Q\rightarrow\mathbb{R}$, where $Q$ is a convex compact, be endowed with $(\delta,L)$-oracle  \cite{devolder2014first}. This means that for some $L  > 0$ at any $y \in Q$ there exists a pair
	$(f_{\delta}(y),g_{\delta}(y))\in \mathbb{R}\times\mathbb{R}^{n}$ such that, for all $y \in Q$, 
	\begin{equation}\label{EEE2}
	f_{\delta}(y)+\langle g_{\delta}(y),x-y\rangle\leq f(x)\leq f_{\delta}(y)+\langle g_{\delta}(y),x-y\rangle+\frac{L\|x-y\|^{2}}{2}+\delta.
	\end{equation}
	
	Let $g(y)$ be any selector of the exact subgradients of $f$. Under an additional assumption that $ \|g_{\delta}(y)-g(y) \| \leq \bar{\delta}_u$, we show that $g_{\delta}(y)$ is an inexact oracle for the operator $g(y)$. This is similar to the exact case, when the subgradient of a convex function defines a monotone operator.
	It is easy to show \eqref{eq:g_or_deff}. Indeed,
	\begin{equation}
	\label{eq:ex4_10s2}
	\langle g_{\delta}(y)-g(y),y-x \rangle \geq -
	\bar{\delta}_u \|x-y\| \geq - \bar{\delta}_u  D, \quad \forall x,y \in Q.
	\end{equation}
	Let us show that \eqref{eq:g_or_def} is satisfied  with $L(\delta_{c})=L,\,\delta_{c}=0$,\\ $\delta_{u}=\max\{2\delta,\bar{\delta}_uD\}$, and $\tilde{g}(y,\delta_c,\delta_u) = g_{\delta}(y)$. Indeed,
	\begin{align}
	 \langle g_{\delta}(z)- g_{\delta}(y),z-x\rangle &= \langle g_{\delta}(y)- g_{\delta}(z),x-z\rangle \notag  \\
	&= \langle g_{\delta}(y),x-y\rangle-\langle g_{\delta}(z),x-z\rangle-\langle g_{\delta}(y),z-y\rangle \notag \\
	&= (f(x)-f_{\delta}(z)-\la g_{\delta}(z),x-z \ra) + (f(z)-f_{\delta}(y) - \la g_{\delta}(y), z-y\ra)  - \notag \\
	&\qquad \qquad \qquad \qquad - (f(x) -f_{\delta}(y) - \la g_{\delta}(y),x-y\ra) + (f_{\delta}(z)-f(z)) \notag \\
	& \leq \left(\frac{L}{2}\|x-z\|^2+\delta\right)+\left(\frac{L}{2}\|z-y\|^2+\delta\right), \notag
	\end{align}
	where in the last inequality we used twice the right inequality in \eqref{EEE2}, and twice the left inequality in \eqref{EEE2}.
\end{Ex}

\section{Generalized Mirror Prox}\label{S_GMP}
In this section, we introduce a new algorithm, which we call Generalized Mirror Prox (GMP), for problem \eqref{eq:PrSt} with inexact oracle for $g$ in the sense of Definition \ref{Def:IO}. The algorithm is listed below as Algorithm \ref{Alg:UMP}.

\begin{algorithm}[ht]
	\caption{Generalized Mirror Prox}
	\label{Alg:UMP}
	\begin{algorithmic}[1]
		\REQUIRE accuracy $\varepsilon > 0$, level $\delta_u >0$ of the uncontrolled oracle error, level $\delta_{pu} >0$ of the uncontrolled error of prox-mapping, initial guess $M_{-1}$ for $L(\delta_c)$, 
		prox-setup: $d(x)$, $V[z] (x)$.
		\STATE Set $k=0$, $z_0 = \arg \min_{u \in Q} d(u)$.
		\FOR{$k=0,1,...$}
		\STATE Set $i_k=0$, $\delta_{c,k} = \frac{\varepsilon}{4}$, $\delta_{pc,k} = \frac{\varepsilon}{8}$.
		\REPEAT
		\STATE Set $M_k=2^{i_k-1}M_{k-1}$.
		\STATE Calculate
		\begin{equation}\label{eq:UMPwStep}
		w_k={\mathop {\arg \min }\limits_{x\in Q}}^{\delta_{pc,k}+\delta_{pu}} \left\{\langle \tilde{g}(z_k,\delta_{c,k},\delta_u),x \rangle + M_kV[z_k](x) 		 \right\}.
		\end{equation}
		\begin{equation}\label{eq:UMPzStep}
		z_{k+1}={\mathop {\arg \min }\limits_{x\in Q}}^{\delta_{pc,k}+\delta_{pu}} \left\{\langle \tilde{g}(w_k,\delta_{c,k},\delta_u),x \rangle + M_kV[z_k](x) 		 \right\}.
		\end{equation}
		\vspace{-1em}
		\STATE $i_k=i_k+1$.
		\UNTIL{
			\begin{equation}\label{eq:UMPCheck}
			\hspace{-2em}\langle \tilde{g}(w_k,\delta_c,\delta_u) - \tilde{g}(z_k,\delta_c,\delta_u), w_k - z_{k+1} \rangle \leq \frac{M_k}{2}\bigl(\|w_k-z_k\|^2 
			+ \|w_k-z_{k+1}\|^2\bigr) +
			\delta_{c,k}+\delta_u.
			\end{equation}}
		\STATE Set $k=k+1$.
		\ENDFOR
		\ENSURE $\widehat{w}_k = \frac{1}{\sum_{i=0}^{k-1}M_i^{-1}}\sum_{i=0}^{k-1}M_i^{-1}w_i$.
	\end{algorithmic}
\end{algorithm}

\begin{Th}\label{Th:UMPGenRate}
	Assume that $g(\cdot)$ and $\tilde{g}(\cdot,\delta_c,\delta_u)$ satisfy \eqref{eq:g_or_def} and \eqref{eq:g_or_deff}. Then, for any $k \geq 1$ and any $u \in Q$,
	\begin{equation*}\label{eq:UMPGenRate}
	\frac{1}{\sum_{i=0}^{k-1}M_i^{-1}} \sum_{i=0}^{k-1} M_i^{-1} \langle g(w_i) , w_i - u \rangle \leq  \frac{1}{\sum_{i=0}^{k-1}M_i^{-1}}  (V[z_0](u)-V[z_{k}](u)) 
	+\frac{\varepsilon}{2}+\delta_u+ 2 \delta_{pu}.
	\end{equation*}
\end{Th}
{\it Proof}
As it follows from \eqref{eq:g_or_def}, if $M_k \geq L\left(\delta_{c,k}\right) =  L\left(\frac{\varepsilon}{4}\right)$, \eqref{eq:UMPCheck} holds. Thus, Algorithm \ref{Alg:UMP} is correctly defined.

Let us fix some iteration $k\geq 0$. For simplicity, we denote \\$\tilde{g}(z_k) = \tilde{g}(z_k,\delta_{c,k},\delta_u) $ and  $\tilde{g}(w_k) = \tilde{g}(w_k,\delta_{c,k},\delta_u)$.
By the definition of inexact prox-mapping \eqref{eq:InPrMap}-\eqref{eq:InPrMap1} and \eqref{eq:UMPwStep}, \eqref{eq:UMPzStep}, we have, for any $u\in Q$,
\begin{align}
& \langle \tilde{g}(z_k) + M_k\nabla d(w_k) - M_k \nabla d(z_k), u - w_k \rangle \geq - \delta_{pu} - \frac{\varepsilon}{8}, \label{eq:Th:UMPGenRatePr1} \\
&\langle \tilde{g}(w_k) + M_k\nabla d(z_{k+1}) - M_k \nabla d(z_k), u - z_{k+1} \rangle \geq - \delta_{pu} - \frac{\varepsilon}{8}. \label{eq:Th:UMPGenRatePr2}
\end{align}
Whence, for all $u\in Q$,
\begin{align*}
\langle \tilde{g}(w_k)&, w_k - u \rangle  =  \langle \tilde{g}(w_k) , z_{k+1} - u \rangle + \langle \tilde{g}(w_k) , w_k - z_{k+1}\rangle \\
& \hspace{-1em} \stackrel{\eqref{eq:Th:UMPGenRatePr2}}{\leq}   M_k \langle \nabla d(z_k) -  \nabla d(z_{k+1}), z_{k+1} - u \rangle + \langle \tilde{g}(w_k) , w_k - z_{k+1}\rangle + \delta_{pu} + \frac{\varepsilon}{8} \\
& \hspace{-1em} =  M_k(d(u) - d(z_k) - \langle \nabla d(z_k), u - z_k \rangle )
- M_k(d(u) - d(z_{k+1}) \\ 
& - \langle \nabla d(z_{k+1}), u - z_{k+1} \rangle) -M_{k}(d\left(z_{k+1}\right)-d\left(z_{k}\right) \\
& -\left\langle\nabla d\left(z_{k}\right), z_{k+1}-z_{k}\right\rangle) + \langle \tilde{g}(w_k) , w_k - z_{k+1}\rangle + \delta_{pu} + \frac{\varepsilon}{8} \\
& \hspace{-1em} =  M_k V[z_k](u)-M_k V[z_{k+1}](u) - M_k V[z_k](z_{k+1}) + \langle \tilde{g}(w_k) , w_k - z_{k+1}\rangle \\
& \hspace{-1em}+ \delta_{pu} + \frac{\varepsilon}{8}
\end{align*}
Further, for all $u\in Q$,
\begin{align*}
\langle \tilde{g}(w_k) , &w_k - z_{k+1}\rangle - M_k V[z_k](z_{k+1})   =  \langle \tilde{g}(w_k) - \tilde{g}(z_k) , w_k - z_{k+1}\rangle \\ 
&- M_k V[z_k](z_{k+1}) + \langle \tilde{g}(z_k) , w_k - z_{k+1}\rangle \\
& \hspace{-1em} \stackrel{\eqref{eq:Th:UMPGenRatePr1}}{\leq}   \langle \tilde{g}(w_k) - \tilde{g}(z_k) , w_k - z_{k+1}\rangle  + M_k \langle \nabla d(z_k) -  \nabla d(w_k), w_k - z_{k+1} \rangle  \\
& - M_k V[z_k](z_{k+1})   + \delta_{pu} + \frac{\varepsilon}{8}  \\
& \hspace{-1em} =  \langle \tilde{g}(w_k) - \tilde{g}(z_k) , w_k - z_{k+1}\rangle + M_k \langle \nabla d(z_k) -  \nabla d(w_k), w_k - z_{k+1} \rangle  \\
&   - M_k(d(z_{k+1})-d(z_k)-\langle \nabla d(z_k), z_{k+1} - z_k \rangle )   + \delta_{pu} + \frac{\varepsilon}{8}  \\
& \hspace{-1em} =  \langle \tilde{g}(w_k) - \tilde{g}(z_k) , w_k - z_{k+1}\rangle- M_k\big(d(w_k)-d(z_k)\\
&-\langle \nabla d(z_k), w_k - z_k \rangle \big) - M_k(d(z_{k+1}) - d(w_k)-\langle \nabla d(w_k), z_{k+1} - w_k \rangle ) \\
&+\delta_{pu} + \frac{\varepsilon}{8}  =  \langle \tilde{g}(w_k) - \tilde{g}(z_k) , w_k - z_{k+1}\rangle  - M_kV[z_k](w_k)\\
&-M_k V[w_k](z_{k+1}) + \delta_{pu} + \frac{\varepsilon}{8} \leq   \langle \tilde{g}(w_k) - \tilde{g}(z_k) , w_k - z_{k+1}\rangle  \\
&-\frac{M_k}{2}(\|z_k-w_k\|^2+\|z_{k+1}-w_k\|^2) + \delta_{pu} + \frac{\varepsilon}{8} \stackrel{\eqref{eq:UMPCheck}}{\leq} \frac{3\varepsilon}{8}+\delta_u + \delta_{pu},
\end{align*}
where we also used that $\delta_{c,k}=\e/4$.

%
Thus, we obtain, for all $u\in Q$ and $i \geq 0$,
$$
M_i^{-1} \langle \tilde{g}(w_i) , w_i - u \rangle \leq  V[z_i](u)-V[z_{i+1}](u) + \frac{M_i^{-1} \varepsilon}{2}+M_i^{-1} (\delta_u + 2 \delta_{pu}).
$$
Summing up these inequalities for $i$ from 0 to $k-1$, we have
\begin{equation*}
\frac{1}{\sum_{i=0}^{k-1}M_i^{-1}} \sum_{i=0}^{k-1} M_i^{-1} \langle \tilde{g}(w_i) , w_i - u \rangle \leq  \frac{1}{\sum_{i=0}^{k-1}M_i^{-1}}  (V[z_0](u)-  V[z_{k}](u))  +\frac{\varepsilon}{2}+\delta_u + 2 \delta_{pu}.
\end{equation*}
By \eqref{eq:g_or_deff}, we obtain the statement of the Theorem \ref{Th:UMPGenRate}.
\qed

Note that the compactness of the set $Q$ was not used in the proof.
\begin{Cor}\label{Cor:VIsol}
	Assume that $g(\cdot)$ and $\tilde{g}(\cdot,\delta_c,\delta_u)$ satisfy \eqref{eq:g_or_def} and \eqref{eq:g_or_deff}. 
	Also let $C\subseteq Q$ be a convex compact.
	Then, for all $k \geq 0$, we have
	\begin{equation}\label{eq:VIsolBound}
	{\rm Gap}_C(\widehat{w}_k) = \max_{u\in C}\langle g(u) , \widehat{w}_k - u \rangle \leq \frac{1}{\sum_{i=0}^{k-1}M_i^{-1}}  \max_{u\in Q}V[z_0](u) + \frac{\varepsilon}{2} +  \delta_u + 2 \delta_{pu},
	\end{equation}
	where
	$     \widehat{w}_k = \left(\sum_{i=0}^{k-1}M_i^{-1}\right)^{-1} \sum_{i=0}^{k-1} M_i^{-1} w_i$.
\end{Cor}
{\it Proof}
By monotonicity of $g$, we have, for all $i \geq 0$ and $u \in Q$,
$$
\langle g(u), w_i - u \rangle = \langle g(w_i), w_i - u \rangle + \langle g(u) - g(w_i), w_i - u \rangle \leq \langle g(w_i), w_i - u \rangle,
$$
Therefore, 
$
\left(\sum_{i=0}^{k-1}M_i^{-1}\right)^{-1}\sum_{i=0}^{k-1} M_i^{-1} \langle g(w_i) , w_i - u \rangle \geq \langle g(u), \widehat{w}_k - u \rangle
$ for any $u \in Q$.
Combining this with Theorem \ref{Th:UMPGenRate} and taking the maximum over all $u \in C$, we obtain the statement of the Corollary \ref{Cor:VIsol}. \qed

If a number $D$ satisfying $\max_{u \in C} V[z_0](u) \leq D$ is known, which is the case for most of the standard proximal setups \cite{nemirovski2004prox}, 
we can construct an adaptive stopping criterion: the algorithm stops whenever 
$
D\left(\sum_{i=0}^{k-1}M_i^{-1}\right)^{-1} \leq \varepsilon/2.
$
This guarantees that the r.h.s. of \eqref{eq:VIsolBound} is no greater than $\varepsilon+  \delta_u + 2 \delta_{pu}$ and $\widehat{w}_k$ is an $(\varepsilon+  \delta_u + 2 \delta_{pu})$-solution to \eqref{eq:PrSt}. 

Next, we consider the case of H\"older-continuous operator $g$ and show that  Algorithm \ref{Alg:UMP} is universal. For simplicity we assume that the prox-mapping is calculated exactly, i.e., $\delta_{pc}=\delta_{pu}=0$ and $\delta_u=0$. In this case, it is sufficient to set $\delta_{c,k} = \frac{\varepsilon}{2}$ at each iteration of Algorithm \ref{Alg:UMP}.
\qed

\begin{Cor}[Universal Method for VI]\label{Cor:UMPHoldUniv}
	Assume that the operator $g$ is H\"older-continuous with constant $L_{\nu}$ for some $\nu \in [0,1]$ and that in Algorithm \ref{Alg:UMP} we have $\delta_{c,k}=\e/2$, $\delta_{u}=0$, $\delta_{pc,k}=0$, $\delta_{pu}=0$. 
	Also let $C\subseteq Q$ be a convex compact.
	Then, for all $k \geq 0$, we have
	\begin{equation}\label{eq:UMPRate}
	{\rm Gap}_C(\widehat{w}_k) = \max_{u\in C} \langle g(u) , \widehat{w}_k - u \rangle \leq \frac{2 L_{\nu}^{\frac{2}{1+\nu}}}{ k \varepsilon^{\frac{1-\nu}{1+\nu}}} \max_{u\in C}V[z_0](u)   + \frac{\varepsilon}{2}.
	\end{equation}
\end{Cor}
{\it Proof}
As it follows from \eqref{eq:g_or_def}, if $M_k \geq L(\delta_{c,k}) L(\frac{\varepsilon}{2})$, \eqref{eq:UMPCheck} holds. Here $L(\cdot)$ is defined in Example \ref{Ex:HoldInexact}. Thus, for all\\ $i = 0, ..., k-1$, we have $M_i \leq 2\cdot L(\frac{\varepsilon}{2})$ and
$$
\frac{1}{\sum_{i=0}^{k-1}M_i^{-1}} \leq \frac{2L(\frac{\varepsilon}{2})}{k} = \frac{2 L_{\nu}^{\frac{2}{1+\nu}}}{ k\varepsilon^{\frac{1-\nu}{1+\nu}}}.
$$
Thus, \eqref{eq:UMPRate} follows from \eqref{eq:VIsolBound}. 
\qed

Let us make several comments on the universal method. Since Algorithm \ref{Alg:UMP} does not use the values of parameters $\nu$ and $L_{\nu}$, we  take the infinum w.r.t. $\nu \in [0,1]$ and obtain the following iteration complexity bound to find $\widehat{w}_k$ satisfying $\max_{u\in C} \langle g(u) , \widehat{w}_k - u \rangle \leq \varepsilon$:
$$
4 \inf_{\nu\in[0,1]}\left(\frac{L_{\nu}}{\varepsilon} \right)^{\frac{2}{1+\nu}} \cdot \max_{u\in C}V[z_0](u).
$$
Using the same reasoning as in \cite{nesterov2015universal}, we estimate the number of oracle calls for Algorithm \ref{Alg:UMP}. The number of oracle calls at each iteration $k$ is equal to $2i_k$, where by $i_k$ we mean the last value of $i_k$ at the end of the inner cycle. So, $M_k=2^{i_k-2}M_{k-1}$ and, hence, $i_k = 2+ \log_2\frac{M_k}{M_{k-1}}$.
Thus, the total number of oracle calls is
\begin{equation}\label{equiv_iter_2}
\sum\limits_{j = 0}^{k-1} i_{j} = 4k + 2\sum\limits_{i = 0}^{k-1}\log_{2} \frac{M_{j}}{M_{j-1}} < 4k + 2 \log_{2} \left(2L\left(\frac{\varepsilon}{2}\right)\right) - 2 \log_{2} (M_{-1}),
\end{equation}
where we used that $M_{k} \leq 2 L(\frac{\varepsilon}{2})$.
Hence, the total number of oracle calls of the Algorithm \ref{Alg:UMP} does not exceed
\begin{equation*}
\inf_{\nu\in[0,1]}\left(16\left(\frac{L_{\nu}}{\varepsilon} \right)^{\frac{2}{1+\nu}} \cdot \max_{u\in C}V[z_0](u) + 2\log_2 2\left(\left(\frac{1}{\varepsilon}\right)^{\frac{1-\nu}{1+\nu}} L_{\nu}^{\frac{2}{1+\nu}}\right)\right)- 2\log_2(M_{-1}).
\end{equation*}

Next, we compare our algorithm with the algorithm in \cite{bach2019universal} which appeared after the first version of this paper appeared as an arxiv preprint. Their algorithm uses AdaGrad-type of stepsizes and, denoting $D^2 \geq \max_{u\in C}V[z_0](u)$, for the non-smooth case $\nu=0$, achieves, up to logarithmic factors, complexity of $O\left(L_0^2D^2\e^{-2}\right)$, which is similar to ours. For the smooth case $\nu=1$ it achieves complexity $O\left((L_0D+L_1D^2)\e^{-1}\right)$, which is similar to ours. Unlike our paper, they consider only two extreme cases $\nu\in\{0,1\}$, but also consider stochastic setting. At the same time, for our algorithm, $1/M_k$ plays the role of stepsize at iteration $k$, and from \eqref{eq:VIsolBound} it is clear that, the smaller $M_k$, the smaller is the right hand side and the better is the convergence guarantee. Step 5 of our algorithm ensures that the stepsize may decrease in the course of the algorithm execution, leading to a better performance in practice. This is in contrast to the stepsize in \cite{bach2019universal} which is a decreasing sequence. 
The experiments in \cite[Appendix 7]{stonyakin2020inexact} show that decreasing $M_k$ (which is equivalent to increasing the stepsize) allows to obtain much faster convergence.

\section{Solving Variational Inequalities with Strongly Monotone Operator}\label{Sec_strongly_VI}
In this section, we assume, that the operator $g$ in \eqref{eq:PrSt} is strongly monotone, which means that, for some $\mu > 0$,
\begin{equation}
\label{eq:gStrMonDef}
\langle g(x) - g(y), x - y\rangle \geq \mu \|x-y\|^2 \quad \forall x, y \in Q.
\end{equation}
We slightly modify the assumptions on the prox-function $d$. Namely, we assume that $0 = \arg \min_{x \in Q} d(x)$ and that $d$ is bounded on the unit ball in the chosen norm $\|\cdot\|$, that is
\begin{equation}
d(x) \leq \frac{\Omega}{2}, \quad \forall x\in Q : \|x \| \leq 1,
\label{eq:dUpBound}
\end{equation}
where $\Omega$ is some known constant. Note that for standard proximal setups, $\Omega = O(\ln \text{dim}E)$ \cite{juditskyrestarts}. Finally, we assume that we are given a starting point $x_0 \in Q$ and a number $R_0 >0$ such that $\| x_0 - x_* \|^2 \leq R_0^2$, where $x_*$ is the solution to \eqref{eq:PrSt}. 
We show that the well-known in optimization restart technique  \cite{nesterov1983method,juditskyrestarts,dvurechensky2016stochastic} also works in the context of VIs. To the best of our knowledge, this is the first time when this technique is applied to VIs. The resulting Restarted Generalized Mirror Prox algorithm is listed below as Algorithm \ref{Alg:RUMP}.

\begin{algorithm}[ht]
	\caption{Generalized Mirror Prox with restarts}
	\label{Alg:RUMP}
	\begin{algorithmic}[1]
		\REQUIRE accuracy $\varepsilon > 0$, $\delta_u >0$, $\delta_{pu} >0$, $\mu >0$, $\Omega$ such that $d(x) \leq \frac{\Omega}{2} \ \forall x\in Q: \|x\| \leq 1$; $x_0, R_0$ such that $\|x_0-x_*\|^2 \leq R_0^2.$
		\STATE Set $p=0,d_0(x)=R_0^2d\left(\frac{x-x_0}{R_0}\right)$.
		\REPEAT
		\STATE Set $x_{p+1}$ as the output of Algorithm \ref{Alg:UMP} for monotone case with accuracy $\mu\varepsilon/2$, $\delta_u$, $\delta_{pu}$, prox-function $d_{p}(\cdot)$ and stopping criterion $\sum_{i=0}^{k-1}M_i^{-1}\geq \frac{\Omega}{\mu}$.
		\STATE Set $R_{p+1}^2 = R_0^2 \cdot 2^{-(p+1)} + 2(1-2^{-(p+1)})(\frac{\varepsilon}{4}+\delta_u+ 2 \delta_{pu})$.
		\STATE Set $d_{p+1}(x) \leftarrow R_{p+1}^2d\left(\frac{x-x_{p+1}}{R_{p+1}}\right)$.
		\STATE Set $p=p+1$.
		\UNTIL			
		$p > \log_2\frac{2R_0^2}{\varepsilon}.$	
		\ENSURE $x_p$.
	\end{algorithmic}
\end{algorithm}

\begin{Th}\label{Th:RUMPCompl}
	Assume that $g$ is strongly monotone with parameter $\mu$. Also assume that the prox-function $d$ satisfies \eqref{eq:dUpBound} and the starting point $x_0 \in Q$ and a number $R_0 >0$ are such that $\| x_0 - x_* \|^2 \leq R_0^2$, where $x_*$ is the solution to \eqref{eq:PrSt}. Then, for $p\geq 0$, the sequence $x_p$ generated by Algorithm \ref{Alg:RUMP} satisfies
	$$
	\|x_p - x_*\|^2 \leq R_0^2\cdot 2^{-p} + \frac{\varepsilon}{2} + \frac{2\delta_u + 4 \delta_{pu}}{\mu},
	$$
	and the point $x_p$ returned by  Algorithm \ref{Alg:RUMP} satisfies $\|x_p-x_*\|^2\leq \varepsilon + \frac{2\delta_u + 4 \delta_{pu}}{\mu}$.
\end{Th}
{\it Proof} 
Let us denote $\Delta = \frac{\varepsilon}{4} + \frac{\delta_u + 2 \delta_{pu}}{\mu}$. We show by induction that, for $p \geq 0$,
$$
\|x_p - x_*\|^2 \leq R_0^2\cdot 2^{-p} + 2(1-2^{-p})\Delta,
$$
which leads to the statement of the Theorem \ref{Th:RUMPCompl}. For $p=0$ this inequality holds by the theorem assumption. Assuming that it holds for some $p\geq 0$, our goal is to prove it for $p+1$ considering the outer iteration $p+1$. Observe that the function $d_{p}$ defined in Algorithm \ref{Alg:RUMP} is 1-strongly convex w.r.t. the norm $\|\cdot\|$.
Using the definition of $d_{p}(\cdot)$ and \eqref{eq:dUpBound}, we have, since $x_p = \arg \min_{x \in Q} d_p(x)$ and $\|x_p-x_*\|\leq R_p$
\begin{equation*}
V_{p}[x_{p}](x_*) = d_{p}(x_{*}) - d_{p}(x_{p}) - \langle \nabla d_{p}(x_{p}), x_{*} - x_{p} \rangle \leq  d_{p}(x_{*})  = R_p^2d\left(\frac{x_{*}-x_p}{R_p}\right) \leq \frac{\Omega R_p^2}{2}.
\end{equation*}
Thus, by  Theorem \ref{Th:UMPGenRate}, taking $u = x_*$, we obtain

\begin{equation*}
\frac{1}{\sum_{i=0}^{k-1}M_i^{-1}} \sum_{i=0}^{k-1} M_i^{-1} \langle g(w_i) , w_i - x_{*} \rangle \leq  \frac{V_{p}[x_{p}](x_{*})}{\sum_{i=0}^{k-1}M_i^{-1}}   + \frac{\mu\varepsilon}{4} + \delta_u + 2 \delta_{pu}  
\leq \frac{\Omega R_p^2}{2\sum_{i=0}^{k-1}M_i^{-1}} + \mu \Delta.
\end{equation*}  

Since the operator $g$ is continuous and monotone, the solution of the Minty VI \eqref{eq:PrSt} is also the solution of the Stampacchia variational inequality \cite{nesterov2007dual,facchinei2007finite}, i.e.,
$\langle g(x_*), x_* - w_i \rangle \leq 0, \quad i=0,...,k-1$.
This and the strong monotonicity of $g$, see \eqref{eq:gStrMonDef}, give, for all $i=0,...,k-1$,
$$
\langle g(w_i) , w_i - x_*\rangle \geq  \langle g(w_i) - g(x_*) , w_i-x_*\rangle \geq \mu\|w_i-x_*\|^2.
$$
Thus, by convexity of the squared norm, we obtain
\begin{align*}
\mu \|x_{p+1}-x_*\|^2 &= \mu \left\|\frac{1}{\sum_{i=0}^{k-1}M_i^{-1}} \sum_{i=0}^{k-1} M_i^{-1} w_i-x_*\right\|^2 \hspace{-0.6em}
\\&
\leq\frac{\mu}{\sum_{i=0}^{k-1}M_i^{-1}} \sum_{i=0}^{k-1} M_i^{-1} \|w_i-x_*\|^2   \\
& \leq \frac{1}{\sum_{i=0}^{k-1}M_i^{-1}} \sum_{i=0}^{k-1} M_i^{-1} \langle g(w_i) , w_i - x_{*} \rangle \leq \frac{\Omega R_p^2}{2\sum_{i=0}^{k-1}M_i^{-1}} + \mu \Delta. 
\end{align*}
Using the stopping criterion $\sum_{i=0}^{k-1} M_i^{-1} \geq \frac{\Omega}{\mu}$, we get
\begin{equation*}
\|x_{p+1}-x_*\|^2 \leq \frac{R_p^2}{2} + \Delta = \frac{1}{2} (R_0^2 \cdot 2^{-p} + 2(1-2^{-p})\Delta ) +\Delta  =
 R_0^2 \cdot 2^{-(p+1)} + 2(1-2^{-(p+1)})\Delta,
\end{equation*}
which finishes the induction proof.
\qed

\begin{Cor}\label{Cor:RUMPHoldUniv}
	Assume that the operator $g$ is H\"older-continuous with constant $L_{\nu}$ for some $\nu \in [0,1]$ and strongly monotone with parameter $\mu$. Then, Algorithm \ref{Alg:RUMP} returns a point $x_p$ such that $\|x_p-x_*\|^2\leq \varepsilon + \frac{2\delta_u+4\delta_{pu}}{\mu}$ and the total number of iterations of the inner Algorithm \ref{Alg:UMP} does not exceed
	\begin{equation}\label{eq_string_monot}
	\inf_{\nu \in [0,1]} \left\lceil  \left(\frac{L_{\nu}}{\mu}\right)^{\frac{2}{1+\nu}}\frac{2^{\frac{2}{1+\nu}}\Omega }{\varepsilon^{\frac{1-\nu}{1+\nu}}} \cdot \log_2 \frac{2 R_0^2}{\varepsilon}\right\rceil.
	\end{equation}
\end{Cor}
{\it Proof}
Let us denote $\hat{p} =  \left\lceil \log_2 \frac{2 R_0^2}{\varepsilon}\right\rceil$. As it was shown in Corollary \ref{Cor:UMPHoldUniv}, at each inner iteration, $M_i \leq 2L(\frac{\mu\varepsilon}{4}) = 2\left(\frac{2}{\mu\varepsilon}\right)^{\frac{1-\nu}{1+\nu}} L_{\nu}^{\frac{2}{1+\nu}}$. Thus, by the stopping criterion $\sum_{i=0}^{k-1} M_i^{-1} \geq \frac{\Omega}{\mu}$, the inner cycle stops at the latest when
$$
k_p = \left\lceil \left(\frac{L_{\nu}}{\mu}\right)^{\frac{2}{1+\nu}}\frac{2^{\frac{2}{1+\nu}}\Omega }{\varepsilon^{\frac{1-\nu}{1+\nu}}}\right\rceil
$$
and we have
\begin{equation*}
N = \sum_{p=1}^{\hat{p}} k_p \leq \left\lceil  \left(\frac{L_{\nu}}{\mu}\right)^{\frac{2}{1+\nu}}\frac{2^{\frac{2}{1+\nu}}\Omega }{\varepsilon^{\frac{1-\nu}{1+\nu}}} \cdot \log_2 \frac{2 R_0^2}{\varepsilon}\right\rceil. 
\end{equation*}
Since Algorithm \ref{Alg:UMP} does not need to know $\nu$ and $L_{\nu}$, we can take the infimum w.r.t. $\nu \in [0,1]$.
\qed
The obtained complexity estimate is optimal for the case $\nu = 1$ \cite{zhang2019lower} and is optimal up to a logarithmic factor for the case $\nu = 0$, \cite{nemirovsky1983problem}. For the intermediate case $\nu \in (0,1)$ we are not aware of any lower bounds. 
As a remark, we note that the complexity estimate for the case $\nu=0$ is $O\left(\frac{L_0^2}{\mu^2\varepsilon}\right)$, whereas one would expect $O\left(\frac{L_0^2}{\mu\varepsilon}\right)$. The reason is that we measure the error in terms of the distance to the solution $\|x_p- x_*\|$, but not in terms of the residual in  VI, i.e. $\max_{u \in Q} \langle g(u), x_p - u \rangle$, as in Corollary \ref{Cor:UMPHoldUniv}.

\section{Applications to Saddle Point Problems}\label{S:saddle_point}
In this section, we consider saddle point problems and show, how Generalized Mirror Prox can be applied to such problems. The problem, we consider is
\begin{equation}\label{eq:SadPointPrSt}
f^*=\min\limits_{u\in Q_1}\max\limits_{v\in Q_2}f(u,v),
\end{equation}
where $Q_1 \subset E_1$ and $Q_2 \subset E_2$ are convex and closed subsets of normed spaces $E_1$ and $E_2$ with norms $\|\cdot\|_1$ and $\|\cdot\|_2$, respectively.
Based on the norms in $E_{1}$ and $E_{2}$, we define the norm on their product $E_1\times E_2$ as\\ $\|x\|=\max\{\|u\|_1,\|v\|_2\}$, $x=(u,v)\in E_1\times E_2$ with the corresponding dual norm $\|s\|_*=\|z\|_{1,*}+\|w\|_{2,*}$, $s=(z,w)\in E^*$,
where $\|\cdot\|_{1,*}$ and $\|\cdot\|_{2,*}$ are the norms on the conjugate spaces  $E^*_1$ and $E^*_2$, dual to $\|\cdot\|_1$ and $\|\cdot\|_2$ respectively.

The function $f$ in \eqref{eq:SadPointPrSt} is assumed to be convex in $u$ and concave in $v$. As it is usually done, we consider the operator
\begin{equation}\label{eq:gSadPointDef}
g(x)=
\begin{pmatrix}
\nabla_u f(u,v)\\
-\nabla_v f(u,v)
\end{pmatrix}, \; x=(u,v) \in Q:=Q_1 \times Q_2.
\end{equation}
By the convexity of $f$ in $u$ and the concavity in $v$, the operator $g$  is monotone
\begin{equation} 
\langle g(x)-g(y),x-y\rangle \geq 0 \quad \forall x,y\in Q\subset E,
\end{equation}
where $x=(u_1,v_1), y=(u_2,v_2)$.



The following lemma gives sufficient conditions for $g$ to be H\"older-continuous, i.e., to satisfy \eqref{eq:GHoeldDef}.
\begin{Lm}\label{Lm:HoeldfToHoeldG}
	Assume that for $f$ in \eqref{eq:SadPointPrSt} there exist a  number \\ $\nu \in [0,1]$ and constants $L_{11,\nu},L_{12,\nu},L_{21,\nu},L_{22,\nu} < +\infty$ such that
	\begin{equation}\label{eq:Sadle_Hold1}
	\|\nabla_u f(u+\Delta u,v+\Delta v)-\nabla_u f(u,v)\|_{1,*}\leq L_{11,\nu}\| \Delta u\|_{1}^{\nu}+L_{12,\nu}\| \Delta v\|_{2}^{\nu},
	\end{equation}
	\begin{equation}\label{eq:Sadle_Hold2}
	\|\nabla_v f(u+\Delta u,v+\Delta v)-\nabla_v f(u,v)\|_{2,*}\leq L_{21,\nu}\| \Delta u\|_{1}^{\nu}+L_{22,\nu}\| \Delta v\|_{2}^{\nu}
	\end{equation}
	for all $u, u+\Delta u\in Q_1, v, v+\Delta v \in Q_2$. Then $g$ defined in \eqref{eq:gSadPointDef} is H\"older-continuous, i.e., satisfies \eqref{eq:GHoeldDef} with the same $\nu$ and $$L_{\nu} = L_{11,\nu}+L_{12,\nu}+L_{21,\nu}+L_{22,\nu}.$$
\end{Lm}
{\it Proof} 
Indeed, for each $x=(u_1,v_1), y=(u_2,v_2) \in Q$ we have:
\begin{align*}
\| g(x)-g(y)\|_* &=  \| \nabla_u f(u_1,v_1)-\nabla_u f(u_2,v_2)\|_{1,*} +\| \nabla_v f(u_1,v_1)-\nabla_v f(u_2,v_2)\|_{2,*} \\&
 \leq L_{11,\nu}\|u_1-u_2\|_1^{\nu}  +L_{12,\nu}\|v_1-v_2\|_2^{\nu}+L_{21,\nu}\|u_1-u_2\|_1^{\nu}+L_{22,\nu}\|v_1-v_2\|_2^{\nu} \\& =(L_{11,\nu}+L_{21,\nu})\|u_1-u_2\|_1^{\nu}+(L_{12,\nu}+L_{22,\nu})\|v_1-v_2\|_2^{\nu}\\&
 \leq(L_{11,\nu}+L_{12,\nu}+L_{21,\nu}+L_{22,\nu}) \max\{\|u_1-u_2\|_1^{\nu},\|v_1-v_2\|_2^{\nu}\}\\&
 =(L_{11,\nu}+L_{12,\nu}+L_{21,\nu}+L_{22,\nu})\|x-y\|^{\nu}.
\qquad \qquad \qquad \qquad \qquad \qquad \quad \qed
\end{align*} 

\begin{Rem}
	As an alternative, one can consider the following primal and dual pair of norms for $E = E_1\times E_2$: $\|x\|=\sqrt{\|u\|_1^2+\|v\|_2^2}$, $x=(u,v)\in E_1\times E_2$, and $\|s\|_*= \sqrt{\|z\|_{1,*}^2+\|w\|_{2,*}^2}$, $s=(z,w)\in E^*$,
	where $\|\cdot\|_{1,*}$ and $\|\cdot\|_{2,*}$ are the norms on the conjugate spaces  $E^*_1$ and $E^*_2$, dual to $\|\cdot\|_1$ and $\|\cdot\|_2$ respectively. 
	We have, for each $x=(u_1,v_1), y=(u_2,v_2) \in Q$,
	\begin{align*}
	\| g(x)-g(y)\|^2_* & = \| \nabla_u f(u_1,v_1)-\nabla_u f(u_2,v_2)\|_{1,*}^2 +\| \nabla_v f(u_1,v_1)-\nabla_v f(u_2,v_2)\|_{2,*}^2 \\&
	\leq 2\big(L_{11,\nu}^2\|u_1-u_2\|_1^{2\nu} +L_{12,\nu}^2\|v_1-v_2\|_2^{2\nu} + L_{21,\nu}^2\|u_1-u_2\|_1^{2\nu} + L_{22,\nu}^2\|v_1-v_2\|_2^{2\nu}\big)
	\\&
	=2(L_{11,\nu}^2+L_{21,\nu}^2)\|u_1-u_2\|_1^{2\nu} + 2(L_{12,\nu}^2+L_{22,\nu}^2)\|v_1-v_2\|_2^{2\nu}
	\\&
	\leq 2(L_{11,\nu}^2 + L_{12,\nu}^2 + L_{21,\nu}^2 + L_{22,\nu}^2)\max\{\|u_1-u_2\|_1^{2\nu},\|v_1-v_2\|_2^{ 2\nu}\}) \\&
	\leq 2(L_{11,\nu}^2+L_{12,\nu}^2+L_{21,\nu}^2+L_{22,\nu}^2)\|x-y\|^{2\nu}
	\end{align*}
	and
	$$
	\| g(x)-g(y)\|_* \leq \sqrt{2(L_{11,\nu}^2+L_{12,\nu}^2+L_{21,\nu}^2+L_{22,\nu}^2)} \|x-y\|^{\nu}.
	$$
\end{Rem}

\begin{Rem}
	Generally speaking, if the set $Q$ is bounded, one can consider different level of smoothness in \eqref{eq:Sadle_Hold1} and \eqref{eq:Sadle_Hold2}.
	More precisely, assume that for some numbers $\nu_{11}, \nu_{12}, \nu_{21}, \nu_{22} \in [0; 1]$, we have
	\begin{align}
	&\|\nabla_u f(u+\Delta u,v+\Delta v)-\nabla_u f(u,v)\|_{1,*}\leq \widehat{L}_{11}\| \Delta u\|_{1}^{\nu_{11}} + \widehat{ L}_{12}\| \Delta v\|_{2}^{\nu_{12}}, \label{eq:Sadle_Hold11} \\
	&\|\nabla_v f(u+\Delta u,v+\Delta v)-\nabla_v f(u,v)\|_{2,*}\leq \widehat{L}_{21}\| \Delta u\|_{1}^{\nu_{21}} + \widehat{L}_{22}\| \Delta v\|_{2}^{\nu_{22}} \label{eq:Sadle_Hold12}
	\end{align}
	for all $u, u+\Delta u\in Q_1, v, v+\Delta v \in Q_2$. Then the statement of  Lemma \ref{Lm:HoeldfToHoeldG} holds for $\nu = \min\{\nu_{11}, \, \nu_{12},\, \nu_{21} \, \nu_{22}\}  \in [0;1].$ Indeed, from \eqref{eq:Sadle_Hold11}, \eqref{eq:Sadle_Hold12}, we have
	\begin{equation*}
	\|\nabla_u f(u+\Delta u,v+\Delta v)-\nabla_u f(u,v)\|_{1,*}\leq \widehat{L}_{11} \cdot D_Q^{\nu_{11} - \nu} \cdot \| \Delta u\|_{1}^{\nu} + \widehat{L}_{12}\cdot D_Q^{\nu_{12} - \nu} \cdot\| \Delta v\|_{2}^{\nu},
	\end{equation*}
	\begin{equation*}
	\|\nabla_v f(u+\Delta u,v+\Delta v)-\nabla_v f(u,v)\|_{2,*}\leq \widehat{L}_{21}\cdot D_Q^{\nu_{21} - \nu} \cdot\| \Delta u\|_{1}^{\nu} + \widehat{L}_{22} \cdot D_Q^{\nu_{22} - \nu} \cdot\| \Delta v\|_{2}^{\nu}
	\end{equation*}
	for all $u, u+\Delta u\in Q_1, v, v+\Delta v \in Q_2$, $D_Q = \sup\{\|x-y\|\,|\, x, y \in Q\}$.
\end{Rem}

The next theorem shows, how  Algorithm \ref{Alg:UMP} can be applied to solve the saddle point problem \eqref{eq:SadPointPrSt}.
\begin{Th}\label{Th:UMPforSadPoint}
	Let the assumptions of Lemma \ref{Lm:HoeldfToHoeldG} hold, the set $Q$ be bounded, and $L_{\nu}$ be given in Lemma \ref{Lm:HoeldfToHoeldG}. Assume also that Algorithm \ref{Alg:UMP} with accuracy $\varepsilon$ is applied to the operator $g$ defined in \eqref{eq:gSadPointDef}. Let $w_i = (u^{i},v^{i})$ be the sequence generated by this algorithm. Then,
	$$
	\max_{v\in Q_2} f(\widehat{u}_k,v) - \min_{u \in Q_1} f(u,\widehat{v}_k) \leq \frac{ 2L_{\nu}^{\frac{2}{1+\nu}}}{ k \varepsilon^{\frac{1-\nu}{1+\nu}}} \max_{x\in Q}V[w_0](x)   + \frac{\varepsilon}{2},
	$$
	$$
	\text{{\rm where}} \quad (\widehat{u}_k, \widehat{v}_k) = \frac{1}{S_k}\sum_{i=0}^{k-1} M_i^{-1} (u^{i},v^{i}), \quad S_k=\sum\limits_{i = 0}^{k-1}  M_i^{-1}.
	$$
	Moreover, in the number of iterations
	$$
	O\left( \inf_{\nu\in[0,1]}\left(\frac{L_{\nu}}{\varepsilon} \right)^{\frac{2}{1+\nu}} \cdot \max_{x\in Q}V[w_0](x) \right),
	$$
	the algorithm finds a pair $(\widehat{u},\widehat{v})$ such that $\max_{v\in Q_2} f(\widehat{u},v) - \min_{u \in Q_1} f(u,\widehat{v}) \leq \varepsilon$.
\end{Th}
{\it Proof.}
By convexity of $f$ in $u$ and concavity of $f$ in $v$, we have, for all $u \in Q_1$,
\begin{align*}
\frac{1}{S_k} \sum_{i=0}^{k-1} \left\langle \nabla_u f(u^{i},v^{i}),u^{i}-u\right \rangle_1 & \geq  \frac{1}{S_k} \sum_{i=0}^{k-1} M_i^{-1} (f(u^{i},v^{i})-f(u,v^{i})) \\
& \geq  \frac{1}{S_k} \sum_{i=0}^{k-1} M_i^{-1} f(u^{i},v^{i})-f(u,\widehat{v}_k).
\end{align*}
In the same way, we obtain, for all $v \in Q_2$,
$$
\frac{1}{S_k} \sum_{i=0}^{k-1} M_i^{-1} \left\langle -\nabla_v f(u^{i},v^{i}),v^{i}-v\right \rangle_2 \geq -\frac{1}{S_k} \sum_{i=0}^{k-1} M_i^{-1} f(u^{i},v^{i})+f(\widehat{u}_k,v).
$$
Summing these inequalities, using \eqref{eq:gSadPointDef} and   Theorem \ref{Th:UMPGenRate}, we obtain that, for all $u \in Q_1$, $v \in Q_2$ and $x = (u,v)$,
$$
f(\widehat{u}_k,v) - f(u,\widehat{v}_k) \leq  \frac{1}{S_k} \sum_{i=0}^{k-1} M_i^{-1} \langle g(w_i), w_i - x \rangle \leq \frac{1}{\sum_{i=0}^{k-1}M_i^{-1}} V[w_0](x) + \frac{\varepsilon}{2}.
$$
Since $M_i \leq 2 L\left(\frac{\varepsilon}{2}\right)$, where $L(\cdot)$ is given in Example \ref{Ex:HoldInexact}, and the set $Q$ is bounded, we obtain
$$
\max_{v\in Q_2} f(\widehat{u}_k,v) - \min_{u \in Q_1} f(u,\widehat{v}_k) \leq \frac{2 L_{\nu}^{\frac{2}{1+\nu}}}{ k \varepsilon^{\frac{1-\nu}{1+\nu}}} \max_{x\in Q}V[w_0](x)   + \frac{\varepsilon}{2}.
$$
The iteration complexity follows from the requirement for $k$ to make the first term in r.h.s. smaller than $\varepsilon$.
\qed

\begin{Rem}
	Since, for a saddle point $(u_*,v_*) \in Q,$ $\max\limits_{v \in Q_2}f(u_*,v)=\min\limits_{u \in Q_1}f(u,v_*)$, the inequality  $\max_{v\in Q_2} f(\widehat{u},v) - \min_{u \in Q_1} f(u,\widehat{v}) \leq \varepsilon$ means that $(\widehat{u},\widehat{v})$ is an $\varepsilon$-optimal solution.
\end{Rem}

\begin{Rem}
	For $\mu$-strongly convex-concave saddle point problems,  Algorithm \ref{Alg:RUMP} returns a point $x_p$ such that $\|x_p-x_*\|^2\leq \varepsilon + \frac{2\delta_u+4\delta{pu}}{\mu}$ (for the exact solution $x_*$) with the complexity estimate given in \eqref{eq_string_monot}.
\end{Rem}

An important particular case of saddle point problem is the Lagrange saddle point problem for a constrained minimization problem.
Let us consider the following convex optimization problem
\begin{equation}
\label{eq:ex4_problem}
\min \{f(x) \;\;:\;\; x \in Q, \quad \phi_j(x)\leq 0, \quad j=1,...,m\},
\end{equation}
where $Q$ is a convex and compact set, $f$ and $\phi_j$ are convex and have H\"older-continuous subgradients
$$
\|\nabla f(x) - \nabla f(y)\|_* \leq L_{\nu_0} \|x - y\|^{\nu_0},
$$ 
$$
\|\nabla \phi_j(x) - \nabla \phi_j(y)\|_* \leq L_{\nu_j} \|x - y\|^{\nu_j} \quad \forall x, y \in Q, j=1,...,m 
$$
for some $\nu_0, ..., \nu_m \in [0,1]$ and $L_{\nu_0}, ..., L_{\nu_m} >0$.
The corresponding Lagrange function for this problem is $L(x,\lambda)= f(x) + \sum\limits_{j=1}^{m}\lambda_j\phi_j(x)$, where $\lambda_j \geq 0$, $j=1,...,m$ are Lagrange multipliers.
If a point $(x_*,\lambda_*)$ is a saddle point of the convex-concave Lagrange function $L(x,\lambda)$, then $x_*$ is a solution to \eqref{eq:ex4_problem}.
Assume also that the Slater's constraint qualification condition holds, i.e., there exists a point $\bar{x}$ such that $\phi_j(\bar{x}) < 0$, $j=1,...,m$. Then it can be shown that the optimal Lagrange multiplier $\lambda_*$ is bounded. Thus, instead of minimization problem  \eqref{eq:ex4_problem}, one can consider the saddle point problem $\min_{x \in Q} \max_{\lambda \in \Lambda} L(x, \lambda)$, which is a convex-concave problem on a bounded set. Using  Lemma \ref{Lm:HoeldfToHoeldG} and the H\"older continuity assumption for the subgradients of $f$ and $\phi_j$, we see that Algorithm \ref{Alg:UMP} and  Theorem \ref{Th:UMPforSadPoint} can be applied. We underline that by its nature, the smoothness level of the primal problem and the dual problem is different. Thus, it is important for the algorithm to adapt to the actual level of smoothness.

Next we introduce the concept inexact oracle for saddle point problems. 

\begin{Def}\label{Def:IOSedlo}
	\textit{
		Assume that for some $\delta_u > 0$ (uncontrolled error) and for any number $\delta_c > 0$ (controlled error) there exists a constant $L(\delta_c) \in ]0, +\infty[$ such that, for any points $x,y \in Q$, one can calculate $\tilde{g}(x,\delta_c,\delta_u)$ and\\ $\tilde{g}(y,\delta_c,\delta_u) \in E^*$ satisfying
		\begin{equation}\label{eq:g_or_def_sedlo}
		\langle \tilde{g}(y,\delta_c,\delta_u) - \tilde{g}(x,\delta_c,\delta_u), y - z \rangle \leq \frac{L(\delta_c)}{2}\left(\|y-x\|^2 + \|y-z\|^2\right) + \delta_c+\delta_u,
		\end{equation}
		and, for each $x = (u_x, v_x), y = (u_y, v_y) \in Q$,
		\begin{equation}\label{eq:g_or_deff_sedlo}
		f(u_y, v_x) - f(u_x, v_y) \leq
		\langle \tilde{g}(y,\delta_c,\delta_u), y - x \rangle + \delta_u.
		\end{equation}
		Then the operator $\tilde{g}(\cdot,\delta_c,\delta_u)$ is called  inexact oracle for the problem \eqref{eq:SadPointPrSt}.}
\end{Def}

\begin{Rem}
	Recall (see Definition \ref{Def:IO}), that $\delta_c$ represents the error of the oracle, which one can control and make as small as we would like to. On the opposite, $\delta_u$ represents the error, which one can not control.
\end{Rem}

\begin{Ex}[\textbf{\it saddle point problems and ($\delta,L$)-oracle in optimization}]
	\label{Ex:delta_L_sad_point}

	Assume that we have access to the operator
	$$g_{\delta}(x)=\begin{pmatrix}
	g_{\delta, u}(u,v)\\
	- g_{\delta, v}(u,v)
	\end{pmatrix}, x=(u,v)\in Q,$$
	where $(f_{\delta, u}, g_{\delta, u})$ is a $(\delta,L)$-oracle for $f$ as a function of $u$, and $-(f_{\delta, v}, g_{\delta, v})$ is a $(\delta,L)$-oracle for ($-f$) as a function of $v$, see \eqref{EEE2}. \\ Define $\|x\| = \|(u, v)\| := \sqrt{\|u\|^2+\|v\|^2}$.
	By Example \ref{ExampVI2.5}, for each
	$x~=~(u_x, v_x)$, $y~=~(u_y, v_y)$, $z~=~(u_z, v_z) \in Q$,
	\begin{align}
	&\langle g_{\delta, u}(u_y, v_y)- g_{\delta, u}(u_x, v_x), u_y-u_z \rangle \leq \frac{L}{2}(\|u_y-u_z\|^2+\|u_y-u_x\|^2)+2\delta, \notag \\ 
	&\langle -g_{\delta, v}(u_y, v_y)+g_{\delta, v}(u_x, v_x), v_y-v_z \rangle \leq \frac{L}{2}(\|v_y-v_z\|^2+\|v_y-v_x\|^2)+2\delta, \notag 
	\end{align}
	and we have
	\begin{equation}\label{eq:ex4_9s1}
	\langle g_{\delta}(y) - g_{\delta}(x),y-z \rangle \leq \frac{L}{2}(\|y-z\|^2+\|y-x\|^2)+4\delta.
	\end{equation}
	Further, from inequalities
	\begin{align}
	&f(u_{y},v_{y})-f(u_{x},v_{y})\leq \langle g_{\delta,u}(y),u_{y}-u_{x}\rangle+\delta, \notag \\
	& f(u_{y},v_{x})-f(u_{y},v_{y})\leq \langle -g_{\delta,v}(y),v_{y}-v_{x}\rangle+\delta \notag
	\end{align}
	we have $
	f(u_{y},v_{x})-f(u_{x},v_{y})\leq\langle g_{\delta}(y),y-x\rangle+2\delta$.
	So, $ \tilde{g}(y,\delta_c,\delta_u)=g_{\delta}(y)$ satisfies  Definition  \ref{Def:IOSedlo}
	with $\delta_{u} = 4 \delta$, $\delta_c = 0$ and $L(\delta_c) = L$.
\end{Ex}	
Similarly to Theorem \ref{Th:UMPforSadPoint} it is sufficient to make 
$$
O\left( \inf_{\nu\in[0,1]}\left(\frac{L}{\varepsilon} \right) \cdot \max_{x\in Q}V[w_0](x) \right)
$$
iterations of Algorithm \ref{Alg:UMP}, to find a pair $(\widehat{u},\widehat{v})$ satisfying
$$
\max_{v\in Q_2} f(\widehat{u},v) - \min_{u \in Q_1} f(u,\widehat{v}) \leq \varepsilon + O(\delta_u +\delta_c).
$$

\section{Conclusions}

In this paper we introduced a definition of inexact oracle for VIs with monotone operator and provided several examples, where such inexactness naturally arises. In particular, we showed, that H\"older-continuous operator is covered by our general framework of inexact oracle. In order to solve VIs with inexact oracle, we generalized Mirror Prox algorithm \cite{nemirovski2004prox} and provided theoretical guarantees for its convergence rate. As a corollary, we proved that this method is universal for VIs with H\"older-continuous monotone operator, i.e., has complexity
$$
O\left( \inf_{\nu\in[0,1]}\left(\frac{L_{\nu}}{\varepsilon} \right)^{\frac{2}{1+\nu}}R^2 \right)
$$
and, unlike existing methods, does not require any knowledge of $L_{\nu}$ or $\nu$. We generalized our algorithm for the case of $\mu$-strongly monotone operators and obtain complexity
$$
O\left( \inf_{\nu \in [0,1]} \left(\frac{L_{\nu}}{\mu}\right)^{\frac{2}{1+\nu}}\frac{1}{\varepsilon^{\frac{1-\nu}{1+\nu}}} \cdot \log_2 \frac{R^2}{\varepsilon}\right)
$$
to find a point $\widehat{x} \in Q$ such that $\|\widehat{x} - x_* \| \leq \varepsilon$. Finally, we showed, how our method can be applied to convex-concave saddle point problems.
In the follow-up work \cite{stonyakin2020inexact} we extended the proposed here methods for the case of inexact relative smoothness and strong convexity.

As a future work it would be interesting to generalize this method for the case of stochastic VIs using the techniques in \cite{dvinskikh2020line-search,gorbunov2021near-optimal} and apply it for the Wasserstein barycenter problem \cite{tiapkin2020stochastic}, apply this method for solving differential games in the spirit of \cite{dvurechensky2015primal-dual,dvurechensky2014algorithms}, extend this algorithm to the case of VIs with operator having higher-order H\"older-continuous derivatives \cite{nesterov2019implementable,gasnikov2019near,bullins2020higher-order,ostroukhov2020tensor}, propose a generalization for zeroth order methods for saddle point problems \cite{gladin2021solving,sadiev2021zeroth-order} using the techniques in \cite{gorbunov2018accelerated,shibaev2021zeroth-order,dvurechensky2021accelerated}, for accelerated methods for saddle-point problems \cite{gasnikov2021accelerated,tominin2021accelerated,mmp_Alkousa_2020,motor_titov_2021}, for decentralized distributed setup by combining with the ideas of \cite{dvinskikh2021decentralized,rogozin2021decentralized,beznosikov2021decentralized}.

\bigskip 

\textbf{Acknowledgements}
	The authors are grateful to Yurii Nesterov for fruitful discussions. The research by P. Dvurechensky and A. Gasnikov in Section \ref{S:IO} was 
	supported by the Ministry of Science and Higher Education of the Russian Federation (Goszadaniye) No. 075-00337-20-03, project No. 0714-2020-0005.
	The research by F. Stonyakin in Section \ref{S:saddle_point} and Appendix B was supported by Russian Science Foundation (project 18-71-00048).

\bibliographystyle{siamplain}
\bibliography{references}

\appendix 
\section*{Appendix A}\label{appA}
Proof of Lemma \ref{Lm:HoldToLip}\\
{\it Proof}
Let us fix some $\nu \in [0,1]$. Then, for any $x \in [0,1 ]$, $x^{2\nu}\leq 1$. On the other hand, for any $x \geq 1$, $\ x^{2\nu}\leq x^{2}$. Thus, for any $x \ge 0$, $x^{2\nu}\leq x^2+ 1$.
Hence, for any $\alpha , \beta \geq 0$,
$$
\alpha^{\nu}\beta\leq \frac{\alpha^{2\nu}}{2}+\frac{\beta^{2}}{2}\leq \frac{\alpha^{2}}{2}+\frac{\beta^{2}}{2}+\frac{1}{2}.
$$
Substituting $\alpha=\frac{ba^{\frac{1}{1+\nu}}}{\delta^{\frac{1}{1+\nu}}}$ and $\beta=\frac{ca^{\frac{1}{1+\nu}}}{\delta^{\frac{1}{1+\nu}}}$, we obtain
$$
\frac{b^{\nu}a^{\frac{\nu}{1+\nu}}}{\delta^{\frac{\nu}{1+\nu}}}\frac{ca^{\frac{1}{1+\nu}}}{\delta^{\frac{1}{1+\nu}}}\leq \frac{b^{2}a^{\frac{2}{1+\nu}}}{2\delta^{\frac{2}{1+\nu}}}+\frac{c^{2}a^{\frac{2}{1+\nu}}}{2\delta^{\frac{2}{1+\nu}}}+\frac{1}{2}
$$
and
$$
ab^{\nu}c \leq \left(\frac{1}{\delta}\right)^{\frac{1-\nu}{1+\nu}} \frac{a^{\frac{2}{1+\nu}}}{2} \left(b^2+c^2\right) + \frac{\delta}{2}.
$$
\qed

\section*{Appendix B}\label{S:numerics}
To show the practical performance of the proposed Algorithm \ref{Alg:UMP}, we performed a series of numerical experiments for the Lagrange saddle point problem induced by the Fermat-Torricelli-Steiner problem.  

All experiments were made using Python 3.4, on a computer with Intel(R) Core(TM) i7-8550U CPU @ 1.80GHz, 1992 Mhz, 4 Core(s), 8 Logical Processor(s), and 8 GB RAM.

We consider an example of a variational inequality with a non-smooth, i.e., with $\nu = 0$, and non-strongly monotone operator. For this VI, the proposed universal method, due to its adaptivity to the smoothness level of the problem, works in practice with iteration complexity much smaller than the one predicted by the theory.
This example is inspired by the well-known {\it Fermat-Torricelli-Steiner problem}, in which we add some non-smooth functional constraints. This problem can be solved by a switching subgradient scheme \cite{polyak1967general,bayandina2018mirror} with complexity $O(1/\varepsilon^2)$, but as we will see, our method allows to obtain much faster convergence in practice than the one given by this bound.

More precisely, for a given set of $N$ points $A_k \in \mathbb{R}^n, k=1,...,N$ consider the optimization problem
$$ \min_{x \in Q} \left\{f(x):= \sum\limits_{k=1}^N \|x - A_k\|_2 \left| \; \varphi_p(x):= \sum_{i=1}^n \alpha_{pi}|x_i| -  1 \leq 0 \right., \; p=1,...,m \right\},$$
where $Q$ is a convex compact, $\alpha_{pi}$ are drawn from the standard normal distribution and then truncated to be positive.
The corresponding Lagrange saddle point problem is defined as 
$$
\min_{x \in Q} \max_{\lambda = (\lambda_1,\lambda_2,\ldots,\lambda_m)^T \in \mathbb{R}^m_+} L(x,\lambda):=f(x)+\sum\limits_{p=1}^m\lambda_p\varphi_p(x),
$$ 
As it was described in \eqref{S:saddle_point}, this problem is equivalent to the variational inequality with monotone non-smooth operator
$$
G(x,\lambda)=
\begin{pmatrix}
\nabla f(x)+\sum\limits_{p=1}^m\lambda_p\nabla\varphi_p(x), \\
(-\varphi_1(x),-\varphi_2(x),\ldots,-\varphi_m(x))^T
\end{pmatrix}.
$$
For simplicity, we assume that there exists (potentially very large) bound for the optimal Lagrange multiplier $\lambda^*$, which allows us to compactify the feasible set for the pair $(x,\lambda)$ to be a Euclidean ball of some radius. We believe that the approach from \cite{monteiro2010complexity,dang2015convergence} to deal with unbounded feasible sets can be extended to our setting and we leave this for future work.

We run Algorithm \ref{Alg:UMP} for different values of  $n, m$, and $N$ with standard Euclidean prox-setup and the starting point $(x^0, \lambda^0) = \frac{1}{\sqrt{m+n}} \textbf{1} \in \mathbb{R}^{n+m}$, where $\textbf{1} $ is the vector of all ones.
The points $A_k$, $k=1,...,N$ are drawn randomly from the standard normal distribution. For each value of the parameters, the random data was drawn 10 times and the results were averaged.
The results of the work of  Algorithm \ref{Alg:UMP} are represented in Fig.  \ref{fig_time_iteration}. For different values of the accuracy $\varepsilon \in \{ 1/2^i, i=1,2,3,4,5,6\}$, we report the number of iterations and the running time in seconds required by Algorithm \ref{Alg:UMP} to reach an $\varepsilon$-solution of the considered problem. 

As it is known \cite{nemirovsky1983problem}, for a VI with a non-smooth operator, the theoretical iteration complexity estimate $O\left(\frac{1}{\varepsilon^2}\right)$ is optimal. However, experimentally we see from slope of the lines in Fig. \ref{fig_time_iteration} that,  due to the adaptivity, the proposed Algorithm \ref{Alg:UMP} has iteration complexity $O\left(\frac{1}{\sqrt[4]{\varepsilon}}\right)$.

\begin{figure}[ht]
	\minipage{0.50\textwidth}
	\includegraphics[width=\linewidth]{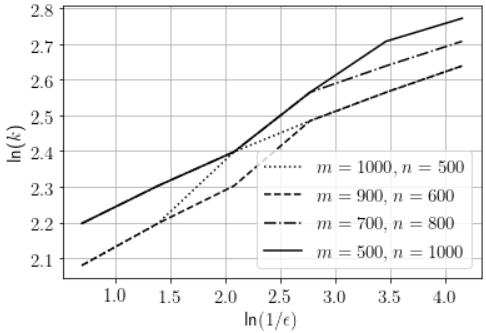}
	\endminipage\hfill
	\minipage{0.50\textwidth}
	\includegraphics[width=\linewidth]{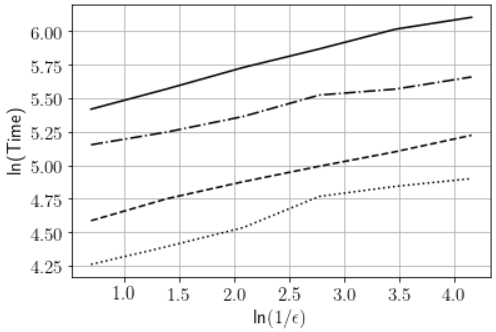}
	\endminipage\hfill
	\caption{Results of  Algorithm \ref{Alg:UMP} for Fermat--Torricelli--Steiner problem  with different values of $m$ and $n$.}
	\label{fig_time_iteration}
\end{figure}

	\section*{Appendix C}\label{numerical_compar}
	In this appendix, in order to demonstrate the performance of the Generalized Mirror Prox with restarts (Algorithm \ref{Alg:RUMP}), we consider the variational inequality with Lipschitz-continuous strongly monotone operator (see Example 5.2 in \cite{khanh2014modified})
	\begin{equation}\label{operator_numerical}
	g: Q \subset \mathbb{R}^n \to \mathbb{R}^n, \quad g(x) = x.
	\end{equation}
	We compare the work of the proposed Algorithm \ref{Alg:RUMP} with Modified Projection Method, which was proposed in \cite{khanh2014modified}. We run Algorithm \ref{Alg:RUMP} with different values of the accuracy $\varepsilon \in \{10^{-i}, i = 3,4,\ldots, 10 \}$ and for the dimension $ n = 10^7$. We take $Q = \{x \in \mathbb{R}^n, \|x\|_2 \leq 2\}$.  The results of the comparison are presented in Fig. \ref{fig_VI}, which illustrates the norm $\|x_{\text{out}} - x_*\|_2$, as a function of iterations, where $x_{\text{out}}$ is the output of each algorithm, and $x_*$ is the solution of the problem \eqref{eq:PrSt}, for the operator \eqref{operator_numerical}. Note that $x_* = \textbf{0} \in \mathbb{R}^n$. In the conducted experiments,  at  the first, we run Algorithm \ref{Alg:RUMP}, and calculate $\|x_{\text{out}} - x_*\|_2$ for the different previously mentioned values of  $\varepsilon$ and the corresponding number of iterations, resulted by the working of algorithm. For the calculated number of iteration of Algorithm \ref{Alg:RUMP}, we rum Modified Projection Method and calculate the corresponding values $\|x_{\text{out}} - x_*\|_2$.  From Fig. \ref{fig_VI}, we can see the higher efficiency  of the proposed Algorithm \ref{Alg:RUMP}, and the big difference between the results of the compared algorithms.

\begin{figure}[ht]
	\centering
	\minipage{0.70\textwidth}
	\includegraphics[width=\linewidth]{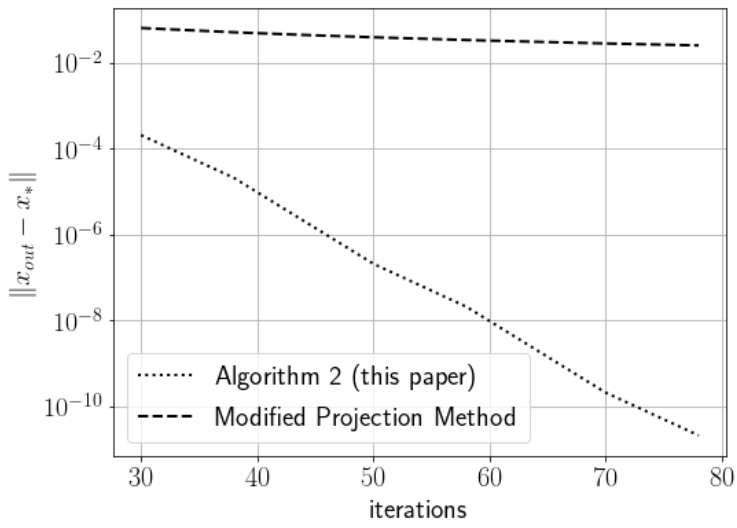}
	\endminipage\hfill
	\caption{Results of  Algorithm \ref{Alg:RUMP} and Modified Projection Method with $n = 10^7$.}
	\label{fig_VI}
\end{figure}

\end{document}

%% file: ex_shared.tex
\usepackage{lipsum}
\usepackage{amsfonts}
\usepackage{graphicx}
\usepackage{epstopdf}
\usepackage{algorithmic}
\usepackage{verbatim}    
\usepackage{multirow}

\ifpdf
  \DeclareGraphicsExtensions{.eps,.pdf,.png,.jpg}
\else
  \DeclareGraphicsExtensions{.eps}
\fi

\newcommand{\email}[1]{\protect\href{mailto:#1}{#1}}

\title{Generalized Mirror Prox Algorithm for Monotone Variational Inequalities: Universality and Inexact Oracle
\thanks{Submitted to the editors DATE.
}}

\makeatletter
\let\@fnsymbol\@arabic
\makeatother

\author{
Fedor Stonyakin
\thanks{V.I. Vernadsky Crimean Federal University, Simferopol; Moscow Institute of Physics and Technology, Moscow
  (\email{fedyor@mail.ru}).}
\and Alexander Gasnikov
\thanks{Moscow Institute of Physics and Technology, Moscow; Institute for Information Transmission Problems RAS, Moscow
	(\email{gasnikov@yandex.ru}).}
\and Pavel Dvurechensky
\thanks{Weierstrass Institute for Applied Analysis and Stochastics, Berlin
	(\email{pavel.dvurechensky@wias-berlin.de}).}	
\and Alexander Titov
\thanks{Moscow Institute of Physics and Technology, Moscow
	(\email{a.a.titov@phystech.edu}).}
\and Mohammad S. Alkousa
\thanks{Moscow Institute of Physics and Technology, Moscow
	(\email{mohammad.alkousa@phystech.edu}).}
}

\usepackage{amsopn}


%% file: UMP_arxiv_2022.bbl
\begin{thebibliography}{10}

\bibitem{mmp_Alkousa_2020}
{\sc M.~S. Alkousa, A.~V. Gasnikov, D.~M. Dvinskikh, D.~A. Kovalev, and F.~S.
  Stonyakin}, {\em Accelerated methods for saddle-point problem}, Computational
  Mathematics and Mathematical Physics, 60 (2020), pp.~1787--1809.

\bibitem{antonakopoulos2021adaptive}
{\sc K.~Antonakopoulos, V.~Belmega, and P.~Mertikopoulos}, {\em Adaptive
  extra-gradient methods for min-max optimization and games}, in International
  Conference on Learning Representations, 2021,
  \url{https://openreview.net/forum?id=R0a0kFI3dJx}.

\bibitem{arjovsky2017wasserstein}
{\sc M.~Arjovsky, S.~Chintala, and L.~Bottou}, {\em Wasserstein {GAN}},
  arXiv:1701.07875,  (2017).

\bibitem{auslender2005interior}
{\sc A.~Auslender and M.~Teboulle}, {\em Interior projection-like methods for
  monotone variational inequalities}, Mathematical programming, 104 (2005),
  pp.~39--68.

\bibitem{aybat2018robust}
{\sc N.~S. Aybat, A.~Fallah, M.~Gurbuzbalaban, and A.~Ozdaglar}, {\em Robust
  accelerated gradient methods for smooth strongly convex functions}, SIAM J.
  Optim., 30 (2020), pp.~717--751.

\bibitem{bach2019universal}
{\sc F.~Bach and K.~Y. Levy}, {\em A universal algorithm for variational
  inequalities adaptive to smoothness and noise}, in Proceedings of the
  Thirty-Second Conference on Learning Theory, A.~Beygelzimer and D.~Hsu, eds.,
  vol.~99 of Proceedings of Machine Learning Research, Phoenix, USA, 25--28 Jun
  2019, PMLR, pp.~164--194,
  \url{http://proceedings.mlr.press/v99/bach19a.html}.
\newblock arXiv:1902.01637.

\bibitem{baimurzina2019universal}
{\sc D.~R. Baimurzina, A.~V. Gasnikov, E.~V. Gasnikova, P.~E. Dvurechensky,
  E.~I. Ershov, M.~B. Kubentaeva, and A.~A. Lagunovskaya}, {\em Universal
  method of searching for equilibria and stochastic equilibria in
  transportation networks}, Computational Mathematics and Mathematical Physics,
  59 (2019), pp.~19--33.

\bibitem{bayandina2018mirror}
{\sc A.~Bayandina, P.~Dvurechensky, A.~Gasnikov, F.~Stonyakin, and A.~Titov},
  {\em Mirror descent and convex optimization problems with non-smooth
  inequality constraints}, in Large-Scale and Distributed Optimization,
  P.~Giselsson and A.~Rantzer, eds., Springer International Publishing, 2018,
  ch.~8, pp.~181--215.

\bibitem{ben-tal2015lectures}
{\sc A.~Ben-Tal and A.~Nemirovski}, {\em Lectures on Modern Convex Optimization
  (Lecture Notes)}, Personal web-page of A. Nemirovski, 2015,
  \url{https://www2.isye.gatech.edu/~nemirovs/LMCO_LN.pdf}.

\bibitem{beznosikov2021decentralized}
{\sc A.~Beznosikov, P.~Dvurechensky, A.~Koloskova, V.~Samokhin, S.~U. Stich,
  and A.~Gasnikov}, {\em Decentralized local stochastic extra-gradient for
  variational inequalities}, arXiv:2106.08315,  (2021).

\bibitem{bogolubsky2016learning}
{\sc L.~Bogolubsky, P.~Dvurechensky, A.~Gasnikov, G.~Gusev, Y.~Nesterov, A.~M.
  Raigorodskii, A.~Tikhonov, and M.~Zhukovskii}, {\em Learning supervised
  pagerank with gradient-based and gradient-free optimization methods}, in
  Advances in Neural Information Processing Systems 29, D.~D. Lee, M.~Sugiyama,
  U.~V. Luxburg, I.~Guyon, and R.~Garnett, eds., Curran Associates, Inc., 2016,
  pp.~4914--4922.
\newblock arXiv:1603.00717.

\bibitem{bullins2020higher-order}
{\sc B.~Bullins and K.~A. Lai}, {\em Higher-order methods for convex-concave
  min-max optimization and monotone variational inequalities},
  arXiv:2007.04528,  (2020).

\bibitem{cohen2018acceleration}
{\sc M.~Cohen, J.~Diakonikolas, and L.~Orecchia}, {\em On acceleration with
  noise-corrupted gradients}, in Proceedings of the 35th International
  Conference on Machine Learning, J.~Dy and A.~Krause, eds., vol.~80 of
  Proceedings of Machine Learning Research, Stockholmsmässan, Stockholm
  Sweden, 10--15 Jul 2018, PMLR, pp.~1019--1028.
\newblock arXiv:1805.12591.

\bibitem{dang2015convergence}
{\sc C.~D. Dang and G.~Lan}, {\em On the convergence properties of
  non-{E}uclidean extragradient methods for variational inequalities with
  generalized monotone operators}, Computational Optimization and Applications,
  60 (2015), pp.~277--310.

\bibitem{aspremont2008smooth}
{\sc A.~d'Aspremont}, {\em Smooth optimization with approximate gradient}, SIAM
  J. on Optimization, 19 (2008), pp.~1171--1183.

\bibitem{devolder2014first}
{\sc O.~Devolder, F.~Glineur, and Y.~Nesterov}, {\em First-order methods of
  smooth convex optimization with inexact oracle}, Mathematical Programming,
  146 (2014), pp.~37--75.

\bibitem{dvinskikh2021decentralized}
{\sc D.~Dvinskikh and A.~Gasnikov}, {\em Decentralized and parallel primal and
  dual accelerated methods for stochastic convex programming problems}, Journal
  of Inverse and Ill-posed Problems, 29 (2021), pp.~385--405,
  \url{https://doi.org/doi:10.1515/jiip-2020-0068},
  \url{https://doi.org/10.1515/jiip-2020-0068}.

\bibitem{dvinskikh2020line-search}
{\sc D.~Dvinskikh, A.~Ogaltsov, A.~Gasnikov, P.~Dvurechensky, and V.~Spokoiny},
  {\em On the line-search gradient methods for stochastic optimization},
  IFAC-PapersOnLine, 53 (2020), pp.~1715--1720,
  \url{https://doi.org/https://doi.org/10.1016/j.ifacol.2020.12.2284}.
\newblock 21th IFAC World Congress, arXiv:1911.08380.

\bibitem{dvurechensky2017gradient}
{\sc P.~Dvurechensky}, {\em Gradient method with inexact oracle for composite
  non-convex optimization}, arXiv:1703.09180,  (2017).

\bibitem{dvurechensky2016stochastic}
{\sc P.~Dvurechensky and A.~Gasnikov}, {\em Stochastic intermediate gradient
  method for convex problems with stochastic inexact oracle}, Journal of
  Optimization Theory and Applications, 171 (2016), pp.~121--145.

\bibitem{dvurechensky2021accelerated}
{\sc P.~Dvurechensky, E.~Gorbunov, and A.~Gasnikov}, {\em An accelerated
  directional derivative method for smooth stochastic convex optimization},
  European Journal of Operational Research, 290 (2021), pp.~601 -- 621,
  \url{https://doi.org/https://doi.org/10.1016/j.ejor.2020.08.027},
  \url{http://www.sciencedirect.com/science/article/pii/S0377221720307402}.

\bibitem{dvurechensky2015primal-dual}
{\sc P.~Dvurechensky, Y.~Nesterov, and V.~Spokoiny}, {\em Primal-dual methods
  for solving infinite-dimensional games}, Journal of Optimization Theory and
  Applications, 166 (2015), pp.~23--51.

\bibitem{dvurechensky2014algorithms}
{\sc P.~E. Dvurechensky and G.~E. Ivanov}, {\em Algorithms for computing
  {M}inkowski operators and their application in differential games},
  Computational Mathematics and Mathematical Physics, 54 (2014), pp.~235--264.

\bibitem{facchinei2007finite}
{\sc F.~Facchinei and J.-S. Pang}, {\em Finite-dimensional variational
  inequalities and complementarity problems}, Springer Science \& Business
  Media, 2007.

\bibitem{gasnikov2019near}
{\sc A.~Gasnikov, P.~Dvurechensky, E.~Gorbunov, E.~Vorontsova,
  D.~Selikhanovych, C.~A. Uribe, B.~Jiang, H.~Wang, S.~Zhang, S.~Bubeck,
  Q.~Jiang, Y.~T. Lee, Y.~Li, and A.~Sidford}, {\em Near optimal methods for
  minimizing convex functions with lipschitz $p$-th derivatives}, in
  Proceedings of the Thirty-Second Conference on Learning Theory,
  A.~Beygelzimer and D.~Hsu, eds., vol.~99 of Proceedings of Machine Learning
  Research, Phoenix, USA, 25--28 Jun 2019, PMLR, pp.~1392--1393.

\bibitem{gasnikov2021accelerated}
{\sc A.~V. Gasnikov, D.~M. Dvinskikh, P.~E. Dvurechensky, D.~I. Kamzolov, V.~V.
  Matyukhin, D.~A. Pasechnyuk, N.~K. Tupitsa, and A.~V. Chernov}, {\em
  Accelerated meta-algorithm for convex optimization problems}, Computational
  Mathematics and Mathematical Physics, 61 (2021), pp.~17--28,
  \url{https://doi.org/10.1134/S096554252101005X},
  \url{https://doi.org/10.1134/S096554252101005X}.

\bibitem{gasnikov2016stochasticInter}
{\sc A.~V. Gasnikov and P.~E. Dvurechensky}, {\em Stochastic intermediate
  gradient method for convex optimization problems}, Doklady Mathematics, 93
  (2016), pp.~148--151.

\bibitem{gasnikov2019adaptive}
{\sc A.~V. Gasnikov, P.~E. Dvurechensky, F.~S. Stonyakin, and A.~A. Titov},
  {\em An adaptive proximal method for variational inequalities}, Computational
  Mathematics and Mathematical Physics, 59 (2019), pp.~836--841.

\bibitem{gasnikov2018universal}
{\sc A.~V. Gasnikov and Y.~E. Nesterov}, {\em Universal method for stochastic
  composite optimization problems}, Computational Mathematics and Mathematical
  Physics, 58 (2018), pp.~48--64.

\bibitem{lan2015generalized}
{\sc S.~Ghadimi, G.~Lan, and H.~Zhang}, {\em Generalized uniformly optimal
  methods for nonlinear programming}, Journal of Scientific Computing, 79
  (2019), pp.~1854--1881.

\bibitem{mintyVI}
{\sc F.~Giannessi}, {\em On {M}inty variational principle}, New Trends in
  Mathematical Programming. Applied Optimization, 13 (1997), pp.~93--99.

\bibitem{gladin2021solving}
{\sc E.~Gladin, A.~Sadiev, A.~Gasnikov, P.~Dvurechensky, A.~Beznosikov, and
  M.~Alkousa}, {\em Solving smooth min-min and min-max problems by mixed oracle
  algorithms}, in Mathematical Optimization Theory and Operations Research:
  Recent Trends, A.~Strekalovsky, Y.~Kochetov, T.~Gruzdeva, and A.~Orlov, eds.,
  Cham, 2021, Springer International Publishing, pp.~19--40.
\newblock arXiv:2103.00434.

\bibitem{gorbunov2021near-optimal}
{\sc E.~Gorbunov, M.~Danilova, I.~Shibaev, P.~Dvurechensky, and A.~Gasnikov},
  {\em Near-optimal high probability complexity bounds for non-smooth
  stochastic optimization with heavy-tailed noise}, arXiv:2106.05958,  (2021).

\bibitem{gorbunov2018accelerated}
{\sc E.~Gorbunov, P.~Dvurechensky, and A.~Gasnikov}, {\em An accelerated method
  for derivative-free smooth stochastic convex optimization}, arXiv:1802.09022,
   (2018).

\bibitem{guminov2017universal}
{\sc S.~Guminov, A.~Gasnikov, A.~Anikin, and A.~Gornov}, {\em A universal
  modification of the linear coupling method}, Optimization Methods and
  Software, 34 (2019), pp.~560--577.

\bibitem{guminov2019accelerated}
{\sc S.~V. Guminov, Y.~E. Nesterov, P.~E. Dvurechensky, and A.~V. Gasnikov},
  {\em Accelerated primal-dual gradient descent with linesearch for convex,
  nonconvex, and nonsmooth optimization problems}, Doklady Mathematics, 99
  (2019), pp.~125--128.

\bibitem{harker1990finite}
{\sc P.~T. Harker and J.-S. Pang}, {\em Finite-dimensional variational
  inequality and nonlinear complementarity problems: a survey of theory,
  algorithms and applications}, Mathematical programming, 48 (1990),
  pp.~161--220.

\bibitem{juditskyrestarts}
{\sc A.~Juditsky and A.~Nemirovski}, {\em First order methods for nonsmooth
  convex large scale optimization, i: general purpose methods}, Optimization
  for Machine Learning,  (2011), pp.~121--148.

\bibitem{dvurechensky2017universal}
{\sc D.~Kamzolov, P.~Dvurechensky, and A.~Gasnikov}, {\em Universal
  intermediate gradient method for convex problems with inexact oracle},
  Optimization Methods and Software,  (2020), pp.~1--28.

\bibitem{khanh2014modified}
{\sc P.~D. Khanh and P.~T. Vuong}, {\em Modified projection method for strongly
  pseudomonotone variational inequalities}, Journal of Global Optimization, 58
  (2014), pp.~341--350.

\bibitem{kniaz2021adversarial}
{\sc V.~V. Kniaz, V.~A. Knyaz, V.~Mizginov, A.~Papazyan, N.~Fomin, and
  L.~Grodzitsky}, {\em Adversarial dataset augmentation using reinforcement
  learning and 3d modeling}, in Advances in Neural Computation, Machine
  Learning, and Cognitive Research IV, B.~Kryzhanovsky, W.~Dunin-Barkowski,
  V.~Redko, and Y.~Tiumentsev, eds., Cham, 2021, Springer International
  Publishing, pp.~316--329.

\bibitem{korpelevich1976extragradient}
{\sc G.~Korpelevich}, {\em The extragradient method for finding saddle points
  and other problems}, Eknomika i Matematicheskie Metody, 12 (1976),
  pp.~747--756.

\bibitem{koshal2011multiuser}
{\sc J.~Koshal, A.~Nedić, and U.~Shanbhag}, {\em Multiuser optimization:
  Distributed algorithms and error analysis}, SIAM Journal on Optimization, 21
  (2011), pp.~1046--1081.

\bibitem{monteiro2010complexity}
{\sc R.~D. Monteiro and B.~F. Svaiter}, {\em On the complexity of the hybrid
  proximal extragradient method for the iterates and the ergodic mean}, SIAM
  Journal on Optimization, 20 (2010), pp.~2755--2787.

\bibitem{nemirovski2004prox}
{\sc A.~Nemirovski}, {\em Prox-method with rate of convergence $o(1/t)$ for
  variational inequalities with lipschitz continuous monotone operators and
  smooth convex-concave saddle point problems}, SIAM Journal on Optimization,
  15 (2004), pp.~229--251.

\bibitem{nemirovsky1983problem}
{\sc A.~Nemirovsky and D.~Yudin}, {\em Problem Complexity and Method Efficiency
  in Optimization}, J. Wiley \& Sons, New York, 1983.

\bibitem{nesterov1983method}
{\sc Y.~Nesterov}, {\em A method of solving a convex programming problem with
  convergence rate $o(1/k^2)$}, Soviet Mathematics Doklady, 27 (1983),
  pp.~372--376.

\bibitem{nesterov2007dual}
{\sc Y.~Nesterov}, {\em Dual extrapolation and its applications to solving
  variational inequalities and related problems}, Mathematical Programming, 109
  (2007), pp.~319--344.
\newblock First appeared in 2003 as CORE discussion paper 2003/68.

\bibitem{nesterov2015universal}
{\sc Y.~Nesterov}, {\em Universal gradient methods for convex optimization
  problems}, Mathematical Programming, 152 (2015), pp.~381--404.

\bibitem{nesterov2019implementable}
{\sc Y.~Nesterov}, {\em Implementable tensor methods in unconstrained convex
  optimization}, Mathematical Programming,  (2019).

\bibitem{nesterov2020primal-dual}
{\sc Y.~Nesterov, A.~Gasnikov, S.~Guminov, and P.~Dvurechensky}, {\em
  Primal-dual accelerated gradient methods with small-dimensional relaxation
  oracle}, Optimization Methods and Software,  (2020), pp.~1--28.

\bibitem{nestrov2011strongly_VI}
{\sc Y.~Nesterov and L.~Scrimali}, {\em Solving strongly monotone variational
  and quasi-variational inequalities}, Discrete \& Continuous Dynamical Systems
  - A, 31 (2011), pp.~1383--1396.

\bibitem{ostroukhov2020tensor}
{\sc P.~Ostroukhov, R.~Kamalov, P.~Dvurechensky, and A.~Gasnikov}, {\em Tensor
  methods for strongly convex strongly concave saddle point problems and
  strongly monotone variational inequalities}, arXiv:2012.15595,  (2020).

\bibitem{ouyang2021lower}
{\sc Y.~Ouyang and Y.~Xu}, {\em Lower complexity bounds of first-order methods
  for convex-concave bilinear saddle-point problems}, Mathematical Programming,
  185 (2021), pp.~1--35, \url{https://doi.org/10.1007/s10107-019-01420-0},
  \url{https://doi.org/10.1007/s10107-019-01420-0}.

\bibitem{polyak1967general}
{\sc B.~Polyak}, {\em A general method of solving extremum problems}, Soviet
  Mathematics Doklady, 8 (1967), pp.~593--597.

\bibitem{rogozin2021decentralized}
{\sc A.~Rogozin, A.~Beznosikov, D.~Dvinskikh, D.~Kovalev, P.~Dvurechensky, and
  A.~Gasnikov}, {\em Decentralized distributed optimization for saddle point
  problems}, arXiv:2102.07758,  (2021).

\bibitem{sadiev2021zeroth-order}
{\sc A.~Sadiev, A.~Beznosikov, P.~Dvurechensky, and A.~Gasnikov}, {\em
  Zeroth-order algorithms for smooth saddle-point problems}, in Mathematical
  Optimization Theory and Operations Research: Recent Trends, A.~Strekalovsky,
  Y.~Kochetov, T.~Gruzdeva, and A.~Orlov, eds., Cham, 2021, Springer
  International Publishing, pp.~71--85.
\newblock arXiv:2009.09908.

\bibitem{shibaev2021zeroth-order}
{\sc I.~Shibaev, P.~Dvurechensky, and A.~Gasnikov}, {\em Zeroth-order methods
  for noisy {H}\"older-gradient functions}, Optimization Letters,  (2021),
  \url{https://doi.org/10.1007/s11590-021-01742-z}.
\newblock (accepted), arXiv:2006.11857.

\bibitem{solodov1999hybrid}
{\sc M.~Solodov and B.~Svaiter}, {\em A hybrid approximate
  extragradient--proximal point algorithm using the enlargement of a maximal
  monotone operator}, Set-Valued Analysis, 7 (1999), pp.~323--345.

\bibitem{stonyakin2020inexact}
{\sc F.~Stonyakin, A.~Tyurin, A.~Gasnikov, P.~Dvurechensky, A.~Agafonov,
  D.~Dvinskikh, M.~Alkousa, D.~Pasechnyuk, S.~Artamonov, and V.~Piskunova},
  {\em Inexact model: A framework for optimization and variational
  inequalities}, Optimization Methods and Software,  (2021),
  \url{https://doi.org/10.1080/10556788.2021.1924714}.
\newblock (accepted), WIAS Preprint No. 2709, arXiv:2001.09013,
  arXiv:1902.00990.

\bibitem{stonyakin2019gradient}
{\sc F.~S. Stonyakin, D.~Dvinskikh, P.~Dvurechensky, A.~Kroshnin,
  O.~Kuznetsova, A.~Agafonov, A.~Gasnikov, A.~Tyurin, C.~A. Uribe,
  D.~Pasechnyuk, and S.~Artamonov}, {\em Gradient methods for problems with
  inexact model of the objective}, in Mathematical Optimization Theory and
  Operations Research, M.~Khachay, Y.~Kochetov, and P.~Pardalos, eds., Cham,
  2019, Springer International Publishing, pp.~97--114.
\newblock arXiv:1902.09001.

\bibitem{tiapkin2020stochastic}
{\sc D.~Tiapkin, A.~Gasnikov, and P.~Dvurechensky}, {\em Stochastic
  saddle-point optimization for wasserstein barycenters}, arXiv:2006.06763,
  (2020).

\bibitem{motor_titov_2021}
{\sc A.~Titov, F.~Stonyakin, M.~Alkousa, and A.~Gasnikov}, {\em Algorithms for
  solving variational inequalities and saddle point problems with some
  generalizations of lipschitz property for operators}, in Mathematical
  Optimization Theory and Operations Research, A.~Strekalovsky, Y.~Kochetov,
  T.~Gruzdeva, and A.~Orlov, eds., Cham, 2021, Springer International
  Publishing, pp.~86--101.

\bibitem{tominin2021accelerated}
{\sc V.~Tominin, Y.~Tominin, E.~Borodich, D.~Kovalev, A.~Gasnikov, and
  P.~Dvurechensky}, {\em On accelerated methods for saddle-point problems with
  composite structure}, arXiv:2103.09344,  (2021).

\bibitem{zhang2019lower}
{\sc J.~Zhang, M.~Hong, and S.~Zhang}, {\em On lower iteration complexity
  bounds for the saddle point problems}, arXiv:1912.07481,  (2019).

\end{thebibliography}
